\documentclass [a4paper,11pt]{amsart}
\usepackage{amsmath,amssymb,amscd,mathrsfs}
\usepackage{enumerate}
\usepackage{epsfig}
\usepackage{scalerel,stackengine}
\usepackage{xcolor}
\usepackage{a4wide}
\allowdisplaybreaks
 \stackMath
\newcommand\reallywidehat[1]{%
\savestack{\tmpbox}{\stretchto{%
  \scaleto{%
    \scalerel*[\widthof{\ensuremath{#1}}]{\kern-.6pt\bigwedge\kern-.6pt}%
    {\rule[-\textheight/2]{1ex}{\textheight}}
  }{\textheight}%
}{0.5ex}}%
\stackon[1pt]{#1}{\tmpbox}%
}

\newcommand\reallywidecheck[1]{%
\savestack{\tmpbox}{\stretchto{%
  \scaleto{%
    \scalerel*[\widthof{\ensuremath{#1}}]{\kern-.6pt\bigwedge\kern-.6pt}%
    {\rule[-\textheight/2]{1ex}{\textheight}}
  }{\textheight}%
}{0.5ex}}%
\stackon[1pt]{#1}{\scalebox{-1}{\tmpbox}}%
}

\DeclareFontFamily{U}{mathx}{}
\DeclareFontShape{U}{mathx}{m}{n}{<-> mathx10}{}
\DeclareSymbolFont{mathx}{U}{mathx}{m}{n}
\DeclareMathAccent{\widecheck}{0}{mathx}{"71}

\DeclareMathOperator*{\bigast}{\scalebox{2.5}{\raisebox{-0.33ex}{$\ast$}}}

\newcommand{\myfrac}[2]{\frac{\raisebox{-2pt}{$#1$}}
	{\raisebox{0.5pt}{$#2$}}}
\newcommand{\ts}{\hspace{0.5pt}}
\newcommand{\exend}{\hfill$\Diamond$}

\numberwithin{equation}{section}

\renewcommand{\leq}{\leqslant}
\renewcommand{\geq}{\geqslant}

\newcommand{\supp}{\mbox{\rm supp}}

\newcommand{\vol}{\mathrm{vol}}
\newcommand{\RR}{{\mathbb R}}
\newcommand{\QQ}{{\mathbb Q}}
\newcommand{\ZZ}{{\mathbb Z}}
\newcommand{\CC}{{\mathbb C}}

\newcommand{\NN}{\mathbb N}
\newcommand{\cL}{{\mathcal L}}

\newcommand{\cF}{{\mathcal F}}
\newcommand{\cM}{{\mathcal M}}
\newcommand{\cS}{{\mathcal S}}

\newcommand{\eps}{\varepsilon}
\newcommand{\vL}{\varLambda}
\newcommand{\ii}{\mathrm{i}}
\newcommand{\ee}{\mathrm{e}}

\newcommand{\dd}{\mbox{\rm d}}

\newcommand{\Cu}{C^{}_{\mathsf{u}}}
\newcommand{\Cc}{C^{}_{\mathsf{c}}}

\newcommand{\Ccinf}{C^{\infty}_{\mathsf{c}}}

\newcommand{\SAP}{\mathcal{SAP}}

\newtheorem{theorem}{Theorem}[section]
\newtheorem{lemma}[theorem]{Lemma}
\newtheorem{proposition}[theorem]{Proposition}
\newtheorem{cor}[theorem]{Corollary}

\newtheorem{definition}[theorem]{Definition} 
\newtheorem{example}[theorem]{Example}
\newtheorem{remark}[theorem]{Remark}
\newtheorem*{thm*}{Theorem}
\newtheorem*{conj*}{Conjecture}

\title{On almost periodicity in crystalline measures}

\author{Jan Maz\'{a}\v{c}}
\author{Christoph Richard}
\address{Department f\"{u}r Mathematik, Friedrich-Alexander-Universit\"{a}t Erlangen-N\"{u}rnberg,
Cauerstrasse 11, 91058 Erlangen, Germany}
\email{jan.mazac@fau.de, christoph.richard@fau.de}
\urladdr{https://sites.google.com/view/mazac-math}

\author{Nicolae Strungaru}
\address{Department of Mathematics and Statistics, MacEwan University, \newline
\hspace*{\parindent}  Edmonton, Alberta, Canada, 
and 
\newline \hspace*{\parindent} 
Institute of Mathematics ``Simon Stoilow'', 
Bucharest, Romania}
\email{strungarun@macewan.ca}
\urladdr{https://sites.google.com/macewan.ca/nicolae-strungaru/}

\begin{document}

\begin{abstract}
Meyer defined crystalline measures as tempered distributions $\mu$ such that both $  \mu  $ and its Fourier transform $\widehat\mu$ are pure-point Radon measures of locally finite support. He conjectured that every crystalline measure is almost periodic as a tempered distribution. Favorov constructed a counterexample and asked whether crystalline measures are at least almost periodic as general distributions.
To resolve Favorov's question, we first show that the almost periodicity of a crystalline measure is characterised in terms of its translation boundedness, in any class of Radon measures, tempered distributions, or general distributions.
We then construct a crystalline Fourier eigenmeasure that fails to be translation bounded even as a distribution. 
We finally construct a crystalline measure that fails to be a~Fourier quasicrystal (in particular, it fails to be slowly increasing), but it is an almost periodic tempered distribution whose Fourier transform is even a norm almost periodic measure.
Our examples fully resolve the questions of Meyer and Favorov and sharply delineate the class boundary of translation boundedness. They also demonstrate the unusual behaviour of crystalline measures beyond the class of Fourier quasicrystals.
\end{abstract}

\maketitle

\section{Introduction}

Aperiodic order has attracted substantial attention over the last few decades. In 2000, Lagarias gave an influential overview of the field \cite{L00}, in which he identified important questions and aspects deserving further study.  To describe crystalline measures within that broader context, we give a brief tour from model sets via Fourier quasicrystals to crystalline measures. This also serves to motivate basic notions and to highlight important connections and recent developments in the field. We will restrict our focus to crystalline measures from Section~\ref{sec:cmintro} onwards, and we present the main results of this paper in Section~\ref{sec:results}.

\subsection{Aperiodic order of finite local complexity and almost periodicity}

For a transparent setting, let us start with the class of point sets $\varLambda$ in Euclidean space such that both~$\varLambda$ and its set of point differences $\varLambda-\varLambda$ are uniformly discrete. It contains all lattices and their shifts. Point sets in this class are locally rigid in the sense that they can only have finitely many local configurations up to shifts, like lattices. The above class is stable under shifting and under taking subsets, which might be regarded as a source of disorder. It came as a surprise that there exist ordered examples in that class beyond finite sums of shifted lattices.  They were found by Meyer in 1972 and are certain projections of subsets of a higher-dimensional lattice \cite{Mey72, Mey12, LRS25}. Meyer's \emph{regular model sets} are, if restricted to be inversion symmetric, the Euclidean examples of \textit{approximate lattices} \cite{BH18}. The above class does not contain any other members than shifted subsets of approximate lattices. 

\smallskip

The order might be described via almost periodicity. Initially being studied for functions \cite{Bohr, B55, E49, L53}, the notion of Bohr almost periodicity was transferred in the 1950's to distributions by Schwartz \cite{Schwartz}. 
Let us call a point set $\varLambda$ \emph{almost periodic} if the function $f^{}_\varphi(x)=\sum_{\lambda\in\varLambda} \varphi(x-\lambda)$ is Bohr almost periodic for every continuous compactly supported function $\varphi$. Note that, for uniformly discrete point sets, such $\varLambda$ is \emph{translation bounded}, that is, $f^{}_\varphi$~is bounded for every $\varphi$. Moreover, almost periodicity enforces relative denseness and positive finite Banach density for any non-empty point set. In the above class, almost periodicity characterises point sets $\varLambda=\varGamma+F$, with $\varGamma$ a lattice and $F$ a finite set \cite[Cor.~5.3.8]{Str17}. Such point sets are sometimes called \emph{ideal crystals}, if they are non-empty. Now, any regular model set can be approximated from above and from below by almost periodic weighted model sets, up to deviations of arbitrarily small density. A~corresponding notion of \emph{generalised almost periodicity} was defined in 2012 by Meyer \cite{Mey12}.  In fact, this type of almost periodicity characterises regular model sets in the above class \cite{LRS25}. On the level of functions, it has been called \emph{Riemann almost periodicity} \cite{H64}. (Such functions arise from Riemann integrable functions on the associated group compactification, in contrast to continuous functions in Bohr almost periodicity.)

\subsection{Pure-point diffraction and quasicrystals}

Another notion of order arises from optics. It was formalised in 1930 by Wiener for the signal analysis of functions \cite{W30} and has been adapted in 1995 to the class of translation-bounded measures with positive lower density by Hof \cite{Hof}. For point sets, one associates to $\varLambda$ an \emph{autocorrelation measure}, which captures information about all frequencies of points in $\varLambda-\varLambda$. Its Fourier transform is called the \emph{diffraction measure}. Then, the pure-pointedness of the diffraction measure is the footprint of order. Note that adding points of zero density to a given point set does not change its diffraction spectrum. In this sense, the above notion of order is weaker than almost periodicity. Regular model sets are pure point diffractive. This central result in diffraction theory may be regarded as a~consequence of the Poisson summation formula for the higher-dimensional lattice (whose importance was already stressed by Meyer, see also \cite{RS17, FCSS} for more recent elaborations). 

\smallskip

In fact, Wiener also studied the connection between pure-point diffractivity and various types of almost periodicity \cite{W30, WW41}. In the context of translation-bounded measures with positive lower density, this question has been independently addressed by Gou\'{e}r\'{e} \cite{Gouere}. There is a~decent understanding of that connection by now \cite{LSS-long}, which is sketched in Remark~\ref{rem:DT} below. We note that almost periodicity implies a pure-point diffraction measure~\cite{LenStr16}. We also remark that pure-point diffraction can be characterised dynamically by the pure-point spectrum of an associated dynamical system. This dynamical approach, which also dates back to Wiener's pioneering works, has proved itself very useful in diffraction theory \cite{LMS02, BL04, BLM07}. We refer to the monograph \cite{TAO} for a detailed account on diffraction and aperiodic order. 

\smallskip

The renewed interest in Wiener's generalised harmonic analysis was triggered by the discovery of certain highly ordered but aperiodic metallic alloys \cite{SBGC84} in diffraction experiments in 1982. They were called \emph{quasicrystals}, and a number of mathematical descriptions were proposed, which mimic constructions of the Penrose tiling \cite{LS84}, see Kramer's review article~\cite{K17}. The corresponding point sets were called cut-and-project sets. In 1986, de Bruijn studied Fourier properties of these point sets \cite{dB86}, and their generalised Poisson summation formula led de Bruijn to coin a notion of \emph{Poisson comb} \cite{dB87}. In 1995, these point sets were identified as regular model sets \cite{M95}. 
  
\subsection{Aperiodic order beyond finite local complexity}

The picture broadens drastically if the above finite local complexity constraint is dropped. The class of uniformly discrete point sets has many almost periodic members beyond ideal crystals, and each of them is pure point diffractive. Simple examples arise from \emph{modulated lattices} with Bohr almost periodic modulation functions, such as $\varLambda=\{3n+\sin(\sqrt{3}n): n\in \ZZ\}$ on the line. In fact, \emph{every} such point set on the line is a modulated lattice, as has recently been shown by Favorov \cite[Thm.~1]{Fav24d}. It was observed already in the 1970's that modulated lattices describe the diffraction of certain ordered solids. They typically have a countable dense set of non-zero Fourier--Bohr coefficients, as in aperiodic regular model sets. (On the other hand, the coefficients are locally absolutely summable, in contrast to those from aperiodic regular model sets.) In 1987, de~Bruijn realised that both modulated crystals and weighted model sets can be described within a common framework \cite{dB87}, which was later called \emph{deformed weighted model sets} \cite{BD00}. (A recent review, which exploits their associated group compactification, is given in \cite{LLRSS}.) As both their (weighted) Dirac comb and its Fourier transform are pure point, they give rise to unusual Poisson summation formulae. Thus, they belong to the class of de Bruijn's Poisson combs \cite{dB87}.

\subsection{Fourier quasicrystals}

After the discovery of modulated crystals and quasicrystals, it was natural to ask whether aperiodic Poisson combs can have a locally finite set of non-zero Fourier--Bohr coefficients. An answer to that question came from a different direction. Almost periodicity was studied for analytic functions already by Bohr (compare \cite{B55}) and was discovered in the 1950's in (locally finite) zero sets of certain entire almost periodic functions by Krein and Levin \cite{L80}. In 1998, Favorov, Rashkovskii, and Ronkin showed that every almost periodic multiset on the line is the zero set of some entire almost periodic function~\cite{FRR98}. In 2020, it came as another surprise that the latter class indeed contains aperiodic uniformly discrete point sets with locally finite Fourier--Bohr spectrum. They were found by Kurasov and Sarnak \cite{KS20, K24} and were named \emph{Fourier quasicrystals}\footnote{This name first appears in Lagarias' overview article \cite{L00} as the title of a section about Fourier properties of uniformly discrete and relatively dense point sets, without further specification. Also note that, apparently, physical relevance of Fourier quasicrystals has not been investigated yet.}. Their Fourier transform is a~measure that fails to be translation bounded but is slowly increasing.

\smallskip

Olevskii and Ulanovski \cite{OU20, OU20b} then showed that, on the line, such point sets are in one-to-one correspondence with zero sets of finite trigonometric polynomials having simple real roots. 
The underlying connection with Lee--Yang polynomials has recently been extended to Dirac combs with integer weights \cite{AV24} and to higher dimensions \cite{AMKV24}. Meyer showed that the above examples also arise from a~generalised cut-and-project construction, which led him to introduce a notion of \emph{curved model set} \cite{Mey23b}. A~similar approach, which uses the group compactification of such point sets together with techniques from algebraic geometry, has been taken up by Lawton and Tsikh to yield examples in general Euclidean space \cite{Law22, LT25}, which overlap with higher-dimensional examples due to Meyer \cite{Mey23, Mey23c}.

\subsection{Crystalline measures}\label{sec:cmintro}

Already in 2016, the question raised in the above section was placed into a broader context. At that time, it had been shown by Lev and Olevskii \cite{LO15} that for a tempered measure~$\mu$,  which is either positive or one-dimensional,
uniformly discrete support of both $\mu$ and $\widehat \mu$ results in a~periodic structure. Meyer defined the class of \emph{crystalline measures}, i.e., tempered distributions $\mu$ such that both $\mu$ and its Fourier transform $\widehat \mu$ are pure-point measures of locally finite support \cite{Mey16}. This class contains the Dirac combs of ideal crystals, whereas aperiodic regular model sets are excluded. Aperiodic examples based on a model set construction (which appear under the name enriched or projective model sets) had been constructed by Lev and Olevskii just before~\cite{LO16}. Meyer gave additional aperiodic examples, based on constructions by Guinand \cite{Guinand}, Kolountzakis \cite{K16}, and by himself. All of these fail to be positive or to have uniformly discrete support. The point sets by Kurasov and Sarnak provided the first examples that are positive \emph{and} do have uniformly discrete support.  

\smallskip

It is an open question whether there exists an aperiodic crystalline measure supported within a model set. For substantial subclasses, the existence of such a measure can be ruled out \cite{LO17, BST20}. Another problem concerns crystalline measures that are Fourier eigenmeasures, such as lattices of density one. In the periodic case on the line, these have been classified~\cite{BSS23}. 
Recently, Fourier eigenmeasures that generalise Guinand's measure to higher dimensions have been found \cite{AKM25} via a~surprising connection to modular forms (which also provides a classification tool). 
Fourier quasicrystals are nowadays defined to be crystalline measures $\mu$ such that both $\mu$ and $\widehat \mu$ are slowly increasing measures. The above examples of crystalline measures are all Fourier quasicrystals. 
Properties of general crystalline measures have been studied notably by Favorov, see e.g.~\cite{Fav24, Fav24b,Fav24c, Fav25}.

\subsection{Almost periodicity in crystalline measures}

Major questions concerning the almost periodicity of general crystalline measures have remained unresolved. If pure-point spectrum and almost periodicity are considered to be Fourier dual notions, crystalline measures should be almost periodic, in an appropriate sense. Initially, Meyer argued that \emph{all} crystalline measures should be almost periodic tempered distributions \cite{Mey16}. He then realised a gap in his argument, made his statement a conjecture~\cite[Conj.~2.1]{Mey17}, and proved it for a~particular subclass of crystalline measures. Favorov proved Meyer's conjecture for all Fourier quasicrystals \cite[Cor.~2]{Fav24b}, and he essentially proved it for all positive crystalline measures \cite[Thm.~1]{Fav24}, as we explain in Section~\ref{sec:WMC} below.
Based on Kolountzakis' constructions of crystalline measures \cite{K16}, Favorov recently gave an intriguing example of a crystalline measure that is not a~Fourier quasicrystal. He showed that it is translation unbounded as a tempered distribution \cite{Fav24c}, and thus it cannot be an almost periodic tempered distribution. This led him to interpret Meyer's conjecture as a statement about almost periodicity as a distribution. 
\smallskip

\subsection{Overview of the article}
\label{sec:results}
In this article, we study almost periodicity in general crystalline measures. We adapt methods from almost periodicity in translation-bounded measures \cite{ARMA, ST16} to the setting of tempered distributions, following \cite{Schwartz, Str17}. We speak of \emph{strong almost periodicity} instead of almost periodicity, in order to distinguish this notion from weak almost periodicity based on Eberlein's weakly almost periodic functions, which generalise Bohr almost periodic functions \cite{E49}.
We show that, for crystalline measures, strong almost periodicity is equivalent to translation boundedness, within the classes of measures, tempered distributions, and general distributions (Theorem~\ref{thm:sappois}). This reduces the question of strong almost periodicity to the question of translation boundedness. Moreover, due to the duality between pure-point spectrum and strong almost periodicity, it suggests crystalline measures $\mu$ such that both $\mu$ and $\widehat \mu$ are translation-bounded tempered distributions as a natural subclass of crystalline measures, which contains all Fourier quasicrystals.

\smallskip

We then re-analyse Favorov's construction. Modified versions using periodic Fourier eigenmeasures from \cite{BSS23} yield two new explicit examples of crystalline measures outside the class of Fourier quasicrystals. Our first example is a Fourier eigenmeasure that is translation unbounded as a distribution (Theorem~\ref{tnm:ex1}). It is therefore \emph{not} strongly almost periodic in any of the above classes. Our second example is a translation-bounded measure whose Fourier transform is a translation-bounded tempered distribution (Theorem~\ref{thm3.41}). In particular, it is strongly almost periodic as a measure (actually, even more, it is a norm almost periodic measure), and its Fourier transform is strongly almost periodic as a tempered distribution. Whereas is does not belong to the class of Fourier quasicrystals, it does not differ that much from them.

\smallskip

Our examples may be considered as extreme cases of crystalline measures outside the class of Fourier quasicrystals, in the sense that translation boundedness is maximally resp.~minimally violated. This leaves the intermediate classes open. Is there a Fourier eigenmeasure that is translation unbounded as a tempered distribution but translation bounded as a distribution? To which class does Favorov's measure belong? 

\medskip

Our paper is structured as follows. In Section 2, we review the relevant notions, with focus on Fourier quasicrystals. In Section~\ref{sec:WMC}, we clarify the relation between strong almost periodicity and translation boundedness.  Section~\ref{sec:FCE} summarises the construction of Favorov's measure. This serves as a preparation for Section~\ref{sec:sapdistr}, where we construct crystalline Fourier measures that are not translation bounded as a distribution. In Section~\ref{sec:NFQC2}, we construct a~crystalline measure that fails to be a Fourier quasicrystal. 
It is translation bounded as a~measure (and norm almost periodic), and its Fourier transform is translation bounded as a~tempered distribution, but fails to be strongly tempered.

\section{Preliminaries}

\subsection{Fourier transformable measures}

We are interested in harmonic analysis and, for that purpose, we use the Schwartz space $\cS(\RR^d)$ as the space of test functions. Our objects are thus considered to be tempered distributions $T \in \cS'(\RR^d)$, the topological dual of the Schwartz space. We also impose the constraint that $T$ can be identified with a complex Radon measure. This excludes examples such as derivatives of delta distributions. Moreover, it imposes a severe local integrability condition, which we discuss further below. We thus assume that there exists a~complex Radon measure $\mu$ on $\RR^d$ such that $\mu(\varphi) = T(\varphi)$ holds for all $\varphi\in\Ccinf(\RR^d)$.

\smallskip

On the other hand, any measure $\mu$ for which there exists a tempered distribution $T^{}_{\mu}$ such that $\mu(\varphi) = T^{}_{\mu}(\varphi)$ holds for all $\varphi\in\Ccinf(\RR^d)$ is called a \emph{tempered measure}.  We simply write $T$ instead of $T_\mu$ since it is always clear from the context, and we usually do not distinguish between $\mu$ and $T$. However, in certain situations, we may do so for clarity of exposition.
When we speak of a measure in the following, we always mean a complex Radon measure, and we denote by $\cM(\RR^d)$ the space of measures on $\RR^d$.

\smallskip

Let us discuss important subclasses of tempered measures. We adopt a geometric viewpoint by specifying the growth rate of their mass on closed balls $B^{}_r(x)$ of radius $r$ about~$x$.
Following~\cite{BS23}, we say that a tempered measure $\mu$ is \emph{strongly tempered} if its total variation measure~$|\mu|$ is a tempered measure. This notion can be characterised in various ways, see~\cite{BS23, Fav24b}, which we summarise below.

\begin{lemma}[strongly tempered measure]\label{lem:st}
   Let $\mu$ be a measure on $\RR^d$. Then the following are equivalent. 
   \begin{itemize}
       \item[(i)] $\mu$ is a tempered measure that is strongly tempered. 
       \item[(ii)] The total variation measure $|\mu|$ is a tempered measure. 
       \item[(iii)] There is some finite number $\alpha$ such that
    \[ |\mu|\bigl(B^{}_{r}(0) \bigr) \,=\, \mathcal{O}(r^\alpha) \qquad  \mbox{as } r\to \infty \ . \]
    \item[(iv)] $\mu$ is a slowly increasing measure, i.e., there exists a polynomial $P$ such that 
    \[ \int^{}_{\RR^d} \myfrac{\dd |\mu|(x)}{1 + |P(x)|} \, < \, \infty \, .\]
    \item[(v)]  For every Schwartz function $\varphi \in \cS(\RR^d)$, we have $\varphi \in L^1\bigl( |\mu|\bigr)$, and 
    \[\varphi \, \longmapsto \, \int^{}_{\RR^d} \varphi(x) \,\dd \mu(x) \] defines a tempered distribution.  \qed
     \end{itemize} 
\end{lemma}

Measures $\mu$ on $\RR^d$ that satisfy property (iii) with $\alpha=d$ constitute a relevant subclass of strongly tempered measures \cite{BST20}, which we call \emph{finite upper density measures}. This subclass contains the translation-bounded measures, which are central objects in diffraction theory~\cite{MoSt}. A~measure $\mu\in \cM(\RR^d)$ is called \emph{translation bounded} if $\mu*\varphi$ is a bounded function for every $\varphi\in \Cc(\RR^d)$. The following characterisation of the class $\cM^\infty(\RR^d)$ of translation-bounded measures on $\RR^d$ combines \cite[Thm.~1.1]{ARMA1}, \cite[Prop.~4.9.21]{MoSt} and \cite[Lem.~3.8]{PRS22}. 
In particular, condition (iii) shows that any translation-bounded measure is indeed a strongly tempered measure.

\begin{lemma}[translation-bounded measure]
\label{lem:tbchar}
   Let $\mu$ be a measure on $\RR^d$. Then the following are equivalent. 
   \begin{itemize}
       \item[(i)] $\mu$ is a translation-bounded measure. 
       \item[(ii)] $|\mu|$ is a translation-bounded measure.
       \item[(iii)] $|\mu|$ has finite upper Banach density, i.e., 
       \[
       \limsup_{r\to\infty} \sup_{x\in \RR^d} \myfrac{1}{\vol(B^{}_r(x))} \ts |\mu|\bigl(B^{}_r(x)\bigr) \,<\, \infty \ .
       \]
        \item[(iv)] For every bounded set $B\subset \RR^d$ having non-empty interior, we have for the $B$-norm 
        \pushQED{\qed}
       \[
        \|\mu\|^{}_{B} \, := \, \sup_{x\in \RR^d}|\mu|(x+B)<\infty \ . \qedhere
       \]
       \popQED
     \end{itemize} 
\end{lemma}

Note that strong temperedness does not imply translation boundedness. A simple counterexample is the measure $\mu = \sum_{k\in \ZZ} k\ts \delta^{}_k$. By Lemma~\ref{lem:st} (iii), the measure $\mu$ is strongly tempered with $\alpha=2$. But $\mu$ clearly has infinite upper density. Thus, by Lemma~\ref{lem:tbchar} (iii), it cannot be translation bounded.

\smallskip

For an arbitrary tempered distribution $T$, the convolution $T*\varphi$ with any Schwartz function~$\varphi$ is a smooth function that grows polynomially \cite[Thm.~2.3.20]{G14}. 
A tempered distribution $T$ is called \emph{translation bounded} if $T*\varphi$ is bounded for every Schwartz function $\varphi$. A tempered measure that is translation bounded as a tempered distribution may fail to be translation bounded as a measure. As noted by Meyer \cite{Mey17}, an example is Guinand's measure, which we will review in Section~\ref{sec:examples}. However, the converse statement holds true.

\begin{lemma}\label{lem:TBMimpTBTD}
Assume that $\mu$ is a translation-bounded measure. Then $\mu$ is translation bounded as a tempered distribution.
\end{lemma}

\begin{proof}
Let $\mu$ be any translation-bounded measure on $\RR^d$. We have already argued that $\mu$ is a~(strongly) tempered distribution. Let $K=[-\frac{1}{2},\frac{1}{2}]^d$ and consider the $K$-norm $\|\mu\|^{}_K <\infty$. Then, by \cite[Lemma~4.5]{SS21}, there exist a finite constant $C$ and a polynomial $P$ of degree $2d$ such that we have for every Schwartz function $\varphi$ the estimate
\[
| (\mu*\varphi) (x) | \,=\, | \mu(\tau^{}_x\varphi^\dagger)| 
= | (\tau^{}_{-x}\mu)(\varphi^\dagger)|
\,\leq C\, \| P \varphi^\dagger \|^{}_\infty \,  \|\tau^{}_{-x} \mu \|^{}_K \,=\, C \ts \| P \varphi^\dagger \|^{}_\infty\,  \| \mu\|^{}_K \,< \, \infty \ .
\]
Here, $\tau_x$ denotes translation by $x$, and $^\dagger$ denotes reflection, that is, $\varphi^\dagger(x) = \varphi(-x)$. 
\end{proof}

We will discuss translation boundedness more systematically in Section~\ref{sec:APimp} below.  In particular, Lemma~\ref{lem:temptb} indicates that the relation between translation boundedness as a~tempered distribution and strong temperedness is subtle. 

\subsection{Fourier transformable point measures}

We now restrict to pure-point measures and choose a setting symmetric under the Fourier transform, which allows us to describe generalised Poisson summation formulae. Such a setting was used in 1987 by de Bruijn \cite{dB87} to describe deformed weighted model sets, with the Gelfand-Shilov space $S^{1/2}_{1/2}$ of distributions~\cite{vE87}. One may instead work in the subspace of tempered distributions, and further restrict to tempered distributions that are also measures. The following notion of Poisson measure is due to Meyer \cite[Def.~5.6]{Mey18}.

\begin{definition}[Poisson measure]
Let $\mu$ be a tempered measure on $\RR^d$ such that both $\mu$ and its distributional Fourier transform $\widehat \mu$ are pure point measures. Then $\mu$ is called a \emph{(tempered) Poisson measure}.
\end{definition}

\begin{remark}[measure assumption]
The measure assumption in the above definition is quite restrictive. As a consequence, any Poisson measure that is translation bounded with respect to the test function space $\Ccinf(\RR^d)$ is strongly almost periodic with respect to this space and its Fourier transform coincides with the Dirac comb weighted by its Fourier--Bohr coefficients, see Lemma~\ref{lem:FBTD} and Theorem~\ref{thm:sappois} below. In this sense, the class of translation-bounded Poisson measures allows for a transparent Fourier theory. 
\end{remark}

Imposing strong temperedness or translation boundedness as a measure on both $\mu$ and $\widehat\mu$ leads to natural subclasses of Poisson measures. The class of Poisson measures that are translation bounded as a measure has been considered in \cite[Def.~1.27]{Mey12}, compare Theorem~\ref{thm:sappois}.

\smallskip

Recall first that a subset of Euclidean space is called \emph{locally finite} if it intersects any compact set in at most finitely many points. If one restricts a Poisson measure to have a locally finite support, all non-trivial accumulation points are excluded. The corresponding notion is due to Meyer \cite[Def.~1]{Mey16}.

\begin{definition}[crystalline measure]
    Assume that both the support $\vL$ of a Poisson measure~$\mu$, and the support $S$ of its Fourier transform $\widehat{\mu}$ are locally finite sets. Then $\mu$ is called a~\emph{(tempered) crystalline measure}. 
\end{definition}

The following lemmas address the structure of the space of crystalline measures and the vague convergence of measures in this setting. The proofs of both statements are obvious, and we omit them.

\begin{lemma}
    The set of all crystalline measures is closed under translations, finite linear combinations, convolution with a measure of finite support, and multiplication by a trigonometric polynomial.  \qed
\end{lemma}

\begin{lemma}
    Let $\vL$ be a locally finite set and let $\mu$ and $\mu^{}_{n}$ be measures supported in $\vL$. Then, 
    \[
    \pushQED{\qed}\mu^{}_{n} \, \xrightarrow{\ v\ } \, \mu  \qquad \Longleftrightarrow \qquad \mu^{}_{n}\bigl(\{\lambda\}\bigr) \longrightarrow \mu\bigl(\{\lambda\}\bigr)  \qquad \mbox{for all }\lambda \in \vL\, . \qedhere \]
    \popQED
\end{lemma}

\smallskip

From a geometric viewpoint, it is natural to consider crystalline measures that are strongly tempered (and hence slowly increasing). The corresponding notion was introduced by Favorov \cite{Fav19}.

\begin{definition}[Fourier quasicrystal]
    Let $\mu$ be a crystalline measure. If both $\mu$ and $\widehat{\mu}$ are strongly tempered measures, then $\mu$ is called a \emph{Fourier quasicrystal}. 
\end{definition}

\begin{remark}[terminology]
Apparently, the name Fourier quasicrystal had been used for some time without a fixed meaning \cite{L00, LO15, Fav16}. The above denomination appears with reference to \cite{LO16}.
This denomination has now been adopted by many authors, and we will follow that convention. One should, however, keep in mind that regular model sets, the mathematical models of physical quasicrystals, do \emph{not} even belong to the class of Poisson measures, since their Fourier transform fails to be a measure when they are non-periodic. Thus, a quasicrystal is never a Fourier quasicrystal, unless it is an ideal crystal. \exend
\end{remark}

\smallskip

We briefly comment on some subclasses of Fourier quasicrystals. A set $\varLambda$ in Euclidean space is called \emph{uniformly discrete} if there is a ball of radius $r>0$ such that $\varLambda \cap B^{}_r(\lambda)=\{\lambda\}$ for all $\lambda\in\varLambda$.
If, for a crystalline measure $\mu$ on the real line, both $\mu$ and $\widehat \mu$ have uniformly discrete support, then both $\mu$ and $\widehat\mu$ have a periodic structure~\cite{LO15}. 
To describe non-periodic examples, various support conditions beyond uniform discreteness have been proposed.
Pure-point measures of finite upper density are studied in \cite{BST20, BSS23}, where they are called \emph{(doubly) sparse measures}. An important class of locally finite sets beyond uniform discreteness and finite upper density has been considered by Favorov \cite{Fav24b}.

\begin{definition}[$p$-discrete set]
    A discrete set $X$ in $\RR^d$ is called \emph{polynomially discrete} (shortly, $p$-discrete), if there are positive numbers $c,h$ such that, for all $x,y \in X$, $x\neq y$, 
    \[|x-y| \, \geqslant \, c\min\bigl\{ 1, |x|^{-h}, |y|^{-h}\bigr\}\, . \]
\end{definition}
Clearly, any uniformly discrete set is $p$-discrete. 
Any $p$-discrete set has restricted behaviour near plus and minus infinity. The points can get arbitrarily close, but the distance between consecutive points cannot drop faster than polynomially. 

\smallskip

The following result is derived using the explicit form of a tempered distribution with locally finite support.

\begin{lemma}[{\cite[Thm.~1]{Fav24b}}]
\label{lem:Fav_p-discrete}
   Let $\mu$ be a tempered measure, such that $\supp(\mu)$ is $p$-discrete. Then $\mu$ is a strongly tempered measure. \qed
\end{lemma}

\begin{example}
Note that strong temperedness does not imply $p$-discrete support. 
As a counterexample, one can consider 
\[
\mu\,=\, \sum_{n=1}^\infty \myfrac{1}{n!} \ts \sum_{k=1}^{n!} \delta^{}_{n+\frac{k}{n!}} \, ,
\]
which is positive, translation bounded (hence strongly tempered) and supported on a locally finite set that is not $p$-discrete.  \exend
\end{example}

Lemma~\ref{lem:Fav_p-discrete} gives the following criterion for a Poisson measure to be a Fourier quasicrystal.
 
\begin{lemma}\label{lem:FQcrit}
    Let $\mu$ be a Poisson measure such that $\supp(\mu)$ and $\supp(\widehat{\mu})$ are both $p$-discrete. Then $\mu$ is a Fourier quasicrystal. \qed
\end{lemma}

\subsection{Strong almost periodicity and translation boundedness}\label{sec:APimp}

In this section, we discuss almost periodicity for distributions, tempered distributions and measures. We use a~general framework that encompasses previous approaches, compare \cite{ARMA, LR07,  MoSt, Schwartz, ST16}. We will stick to the terminology by  Gil de Lamadrid and Argabright \cite{ARMA}. Instead of speaking of almost periodicity based on Bohr almost periodic functions, we use the name \emph{strong almost periodicity}, in order to distinguish the above notion from weak almost periodicity, which relies on weakly almost periodic functions that generalise the Bohr almost periodic functions \cite{E49, ARMA}. 

\smallskip

Fix a locally compact abelian group $G$, assumed to be $\sigma$-compact for convenience, and consider the vector space $\Cu(G)$ of uniformly continuous and bounded functions on $G$. Denote by $SAP(G)$ the subspace of Bohr almost periodic functions \cite{Bohr,MoSt}. Recall that any Bohr almost periodic function $f\in SAP(G)$ is amenable, and its mean $M(f)$ may be evaluated on any  F\o lner sequence  $(A^{}_n)^{}_{n\in \NN}$ via  
\[
M(f) \, := \,\lim_{n\to\infty} \myfrac{1}{m^{}_G (A^{}_n)} \int_{A^{}_n}  f(x) \, \dd x \ ,
\]
where $m^{}_G$ denotes a Haar measure on $G$, compare \cite[Sec.~4.5]{MoSt}.
In Euclidean space, one might use cubes $A^{}_n=[-n,n]^d$ or $n$-balls. In fact, for Bohr almost periodic functions, the choice of F\o lner sequence does not play any role, and the mean is independent of it. 

\smallskip

Bohr almost periodic functions are stable under pointwise multiplication, in particular, by a character. Thus, for any $f\in SAP(G)$ and arbitrary character $\chi \in \widehat{G}$, the function $x\mapsto  \overline{\chi(x)} \ts \ts f(x)$ is Bohr almost periodic, and its mean is well defined. The number
\[
 \mathbf{a}_\chi^{}(f)\,:=\,M(\overline{\chi}\ts f) \, = \, \lim_{n\to \infty } \myfrac{1}{m^{}_G (A^{}_n)} \int_{A^{}_n}  \overline{\chi(x)} f(x)\, \dd x
\]
is called the \emph{Fourier--Bohr coefficient} of $f$ at $\chi$. Whereas Fourier analysis of a general Bohr almost periodic function is involved (see e.g.~\cite[Ch.~I.8]{B55}), simplifications arise if its Fourier transform happens to be a measure. The following result has been formulated by Meyer for Bohr almost periodic functions on the line. For later reference, we state it here for Euclidean space and remark that it also holds in the general setting.

\begin{lemma}[Fourier transform of a Bohr almost periodic function {\cite[Thm.~3.8]{Mey18}}]\label{CharSAP}
Let $f\in SAP(\RR^d)$ and consider its Fourier--Bohr spectrum $S=\{k\in \RR^d: \mathbf{a}_{k}^{}(f)\ne 0 \}$, which is at most countable. Then, the following statements are equivalent.
\begin{itemize}
\item[(i)] The distributional Fourier transform $\widehat f\in \cS'(\RR^d)$ of $f$ is a measure.
\item[(ii)] For every compact $K\subset \RR^d$, the sum $\sum_{k\in S\cap K} |\mathbf{a}^{}_{k}(f)|$ is finite.
\end{itemize}
If any of the above statements holds, then the identity $\widehat f=\sum_{k\in S} \mathbf{a}^{}_{k}(f)\delta^{}_{k}$ holds in the measure sense.
\qed
\end{lemma}

\smallskip

Now, we transport the notion of almost periodicity to (tempered) distributions and measures via convolution with appropriate test functions, compare \cite[Ch.~VI §9]{Schwartz}, \cite[Cor.~5.5]{ARMA} or \cite[Sec.~4]{LR07}.

\begin{definition}[strong almost periodicity and translation boundedness]
    Fix some locally compact abelian group $G$ (in additive notation), and let $\mathcal{X}$ be a topological vector space of functions on $G$. Denote by $\mathcal{Y}$ the space of all continuous linear functionals on $\mathcal{X}$, the topological dual $\mathcal{X}'$ of~$\mathcal{X}$. Suppose that $G$ acts on $\mathcal{X}$ continuously via translations $\tau^{}_{t}$. For $x\in \mathcal{X}$ and $Y\in \mathcal{Y}$, consider the function  
    $Y\ast x: G\to\CC$ defined by
    \[ 
    Y\ast x \ :\, t \, \longmapsto \, \langle \, Y\, , \, \tau^{}_{t}\ts x^{\dagger}\, \rangle.
    \]
    Here, $^{\dagger}$ denotes the reflection, and we use $\langle \cdot \, , \, \cdot \rangle$ for evaluation of the functional on a function. 
\begin{itemize}
  \item[(i)]  We say that $Y$ is \emph{translation bounded} if $ Y\ast x$ is a bounded function for every $x\in \mathcal{X}$. The subspace of translation-bounded linear functionals will be denoted $\mathcal{Y}^{}_{\infty}$.
  \item[(ii)] We call $Y$ \emph{strongly almost periodic} if the function $ Y\ast x$ is a Bohr almost periodic function for all $x\in \mathcal{X}$. The subspace of strongly almost periodic linear functionals will be denoted $\mathsf{SAP}^{}_{\mathcal{X}}(G) \subset \mathcal{Y}^{}_{\infty}$. 
\end{itemize}
\end{definition}

Let us first note the following consequence for $\mathcal{X}^{}_{1}$ being a dense subspace of $\mathcal{X}^{}_{2}$, such that the identity map $i:\mathcal{X}^{}_{1}\to \mathcal{X}^{}_{2}$ is continuous. In that case, $\mathcal{Y}^{}_2=\mathcal{X}'_2$ is a subspace of $\mathcal{Y}^{}_1=\mathcal{X}'_1$, compare the argument in~\cite[Thm.~4.12]{Rudin}.
We thus have inclusions $\mathcal{Y}^{}_{2,\infty}\subset \mathcal{Y}^{}_{1,\infty}$ and $\mathsf{SAP}^{}_{\mathcal{X}^{}_2}(G)\subset \mathsf{SAP}^{}_{\mathcal{X}^{}_1}(G)$.

\begin{remark}[almost periodicity in diffraction theory \cite{LSS-long}]\label{rem:DT}
From a physical perspective, a~translation-bounded measure $\mu$ is a natural description of an infinite piece of matter. This leads to the setting $\mathcal{X}=\Cc(G)$ and $\mathcal{Y}=\cM(G)$. Any strongly almost periodic measure $\mu$ has a~pure-point diffraction measure \cite{LR07}. In general, pure-point diffractivity of $\mu$ is characterised by strong almost periodicity of its associated autocorrelation measure $\gamma^{}_{\mu}$\/, compare  Proposition~\ref{thm:SAPTBM} below. This, in turn, is equivalent to mean almost periodicity of $\mu$ \cite{LSS-long}. An important subclass comprises Besicovitch almost periodic measures. It is characterised by the property that their diffraction measure is pure point, and it has the squared moduli of their Fourier--Bohr coefficients as weights \cite{Gouere, LSS-long}. 
In diffraction theory, the notation $\cM^\infty(G)$ is used for translation-bounded measures, and the notation $\SAP(G)$ is used for strongly almost periodic measures~\cite{MoSt}. Here, we will write $\cM^\infty(G)$ instead of $\cM_\infty$, and we will write $\mathsf{SAP}^{}_{\Cc}(G)$ instead of $\SAP(G)$.
\exend
\end{remark}

In this article, we will work on $G=\RR^d$, and we will consider the following spaces equipped with their usual topologies: 
\begin{itemize}
\item[(i)] $\mathcal{X}=\Cc(\RR^d)$ and $\mathcal{Y}=\cM(\RR^d)$, the measures on $\RR^d$,
    \item[(ii)] $\mathcal{X}=\Ccinf(\RR^d)$ and $\mathcal{Y}=\mathcal D(\RR^d)$, the distributions on $\RR^d$,
    \item[(iii)] $\mathcal{X}=\cS(\RR^d)$, the Schwartz functions, and $\mathcal{Y}=\cS'(\RR^d)$, the tempered distributions. 
\end{itemize}

Recall that $\Ccinf(\RR^d)$ is a dense subspace of $\cS(\RR^d)$ and that the identity map is continuous, see e.g.~\cite[Thm.~7.10]{Rudin}.
Also observe that, for all of the above spaces, the function $Y\ast x$ is the usual convolution of $Y$ with a test function $x$. 

\begin{remark}[translation-bounded measures and the Wiener algebra]
By the above argument, translation-boundedness as a tempered distribution implies translation-boundedness as a distribution. By Lemma~\ref{lem:TBMimpTBTD}, translation-boundedness as a measure implies translation-boundedness as a tempered distribution. In fact, the latter statement already follows from considering the Wiener algebra $\mathcal X=W(\RR^d)$ and its dual $\mathcal Y=W'(\RR^d)=\cM^\infty(\RR^d)$, the translation-bounded measures \cite{Letal74}, see also \cite{FCSS} for a recent review. We have that $\cS(\RR^d)$ is a~dense subspace of $\mathcal X$, with continuous inclusion.  As a consequence, strong almost periodicity as a translation-bounded measure implies strong almost periodicity as a tempered distribution, which in turn implies strong almost periodicity as a distribution.
\exend
\end{remark}

The following result is taken from \cite{ST16} and clarifies strong almost periodicity for tempered distributions in relation to strong almost periodicity for general distributions. It will later be used for the proof of Theorem~\ref{prop:M-theorem}. We provide a~short argument that will reappear in Proposition~\ref{lem:SAP3}.

\begin{lemma}[{\cite[Lem.~5.2]{ST16}}]\label{lem:SAPchar}
For $T\in \mathcal S'(\RR^d)$, we have $T\in \mathsf{SAP}^{}_{\cS}(\RR^d)$ if and only if \/$T\in \cS'_\infty(\RR^d)$ and for all $\varphi\in \Ccinf(\RR^d)$ we have $T*\varphi\in SAP(\RR^d)$. 
\end{lemma}

\noindent For its proof, we recall the notion of approximate identity, see e.g.~\cite[Sec.~1.2.4]{G14}.

\begin{remark}[approximate identity]\label{rem:appid}
 Let $\varphi\in \cS(\RR^d)$ be any non-negative Schwartz function such that $\|\varphi\|^{}_1=1$, and define $\varphi^{}_n(x)=n\varphi(nx)$. Then $(\varphi^{}_n)^{}_{n\in \NN}$ is an approximate identity for Banach algebra $(\Cu(\RR^d), \|\cdot\|^{}_\infty)$, i.e.,  $(f*\varphi^{}_n)^{}_{n\in\NN}$ converges uniformly to $f$, for every $f\in \Cu(\RR^d)$. For example, one may choose $\varphi\in \Ccinf(\RR^d)$ to obtain an approximate identity of compactly supported Schwartz functions. More details about approximate identities in Schwartz spaces can be found, for example, in \cite{AAL21,Gra08}. \exend
\end{remark}

We also recall that Bohr almost periodicity is stable under convolution with integrable functions.

\begin{remark}[$SAP$ is convolution stable]\label{rem:SAPstable}
Note that, for $f\in SAP(\RR^d)$ and $\varphi\in L^1(\RR^d)$, we have $f*\varphi \in SAP(\RR^d)$. 
Indeed, we have the standard estimate
\[
\|f*\varphi\|^{}_\infty \, = \, \sup_{x\in \RR^d} \Big|\int_{\RR^d} f(x-y)\varphi(y) \,\dd y\Big| \, \leq \, \|f\|^{}_\infty\cdot \|\varphi\|^{}_1   \, ,
\]
which shows that the convolution is bounded. From the same estimate, one also deduces the uniform continuity of the convolution. Moreover, it can be used to relate the almost periods of the convolution to those of $f$, thereby showing Bohr almost periodicity. For the mean of $f\ast \varphi$, we have $M(f\ast \varphi)=M(f)\cdot \widehat{\varphi}(0)$, compare \cite[Prop.~4.5.10]{MoSt}. 
\exend
\end{remark}

\begin{proof}[Proof of Lemma~\ref{lem:SAPchar}]
The ``only if"-part is trivial. For the ``if"-part, consider any $f\in \cS(\RR^d)$. Take an approximate identity $(\varphi^{}_n)^{}_{n\in \NN}$ of compactly supported Schwartz functions as in Remark~\ref{rem:appid}. Define $f^{}_n=f* \varphi^{}_n$ and note that $T*\varphi^{}_n\in SAP(\RR^d)$ by assumption. This implies $T*f_n\in SAP(\RR^d)$ by Remark~\ref{rem:SAPstable}.
Moreover, we have $T*f\in \Cu(\RR^d)$ by assumption. As $T*f^{}_n=(T*f)*\varphi^{}_n$, we also have that $(T*f^{}_n)^{}_{n\in\NN}$ converges to $T*f$ uniformly. As $SAP(\RR^d)$ is closed with respect to the supremum norm, we conclude $T*f\in SAP(\RR^d)$.
\end{proof}

\smallskip

For simplicity, we now restrict our discussion to tempered measures that are strongly tempered, an assumption that is often found in the literature. General tempered measures will be discussed in the following section.

\begin{lemma}\label{lem:temptb}
Let $T$ be a strongly tempered measure. Then $\widehat T$ is a translation-bounded tempered distribution.
\end{lemma}

\begin{proof}
Consider any $\psi\in \cS(\RR^d)$ and write $\varphi=\widecheck{\psi}\in \cS(\RR^d)$. As $T$ is a strongly tempered measure and as $\varphi$ is decaying faster than any polynomial, the measure $T\cdot\varphi$ is finite. Thus $\widehat T*\psi=\widehat T*\widehat \varphi=\widehat{T\cdot\varphi}$ is the Fourier transform of a~finite measure and hence bounded.
\end{proof}

The next lemma appears in \cite[Lem.~1]{Fav19} and in \cite[Thm.~2]{Fav24b}. We provide a short proof for a self-contained presentation.

\begin{lemma}
\label{lem:STimpliesTB}
Let $T$ be a strongly tempered measure. If $T$ is pure point, then $\widehat T$ is a strongly almost periodic tempered distribution.
\end{lemma}

\begin{proof}
We can follow the proof of Lemma~\ref{lem:temptb}. As $T$ is pure point, the finite measure $T\cdot \varphi$ is also pure point, and Bohr almost periodicity of its Fourier transform follows, see e.g.~\cite[Lem.~4.8.10]{MoSt}.
\end{proof}

As an immediate consequence, we recover the following result.
\begin{cor}[{\cite[Cor.~2]{Fav24b}}]\label{theo:Fav24b}
    Let $\mu$ be a Poisson measure and suppose that its distributional Fourier transform $\widehat{\mu}$ is a strongly tempered measure. Then, $\mu$ is a strongly almost periodic tempered distribution. \qed
\end{cor}

This implies that all Fourier quasicrystals are strongly almost periodic tempered distributions.
Combining Lemma~\ref{lem:Fav_p-discrete} with Lemma~\ref{lem:STimpliesTB} gives the following result by Favorov for measures with $p$-discrete support. 

\begin{lemma}[{\cite[Thm.~5]{Fav24b}}]\label{thm:Fav-sap}
   Let $\mu$ be a tempered measure, such that $\supp(\mu)$
   is $p$-discrete. Then $\widehat \mu$ is a strongly almost periodic tempered distribution. \qed
\end{lemma}

The converse of Lemma~\ref{thm:Fav-sap} does not hold. One can construct a measure supported on a~set that is not $p$-discrete, whose Fourier transform is even a norm almost periodic measure. We construct such a measure later in Theorem~\ref{thm3.41}.

\smallskip

On the other hand, if \emph{every} tempered distribution supported in a given fixed locally finite set $\vL$ has a Fourier transform that is strongly almost periodic as a tempered distribution, then the set $\vL$ is $p$-discrete, as shown by Favorov \cite[Thm.~5]{Fav24b}.

\subsection{Some pure-point Fourier eigenmeasures}\label{sec:examples}

Here we discuss illustrative examples of pure-point Fourier eigenmeasures on the line. Our first one is a Poisson measure arising from a cut-and-project construction, which appears in \cite[Thm.~5.3]{BSS23}. For the Dirac comb of a~weighted model set, aperiodicity enforces its Fourier transform to have maximal support. One may modify its weight functions to get invariant under the Fourier transform. This leads to a dense Dirac comb \cite{Rich03, LR07}.  
As our example is derived from a cut-and-project construction, it is translation bounded and, as such, also strongly tempered. As it has maximal support (and so does its Fourier transform), it is not a Fourier quasicrystal (and not a crystalline measure either). 

\begin{example}[dense Dirac comb]
Consider $h(x) = \ee^{-\pi x^2}$, an eigenfunction for the Fourier transform. Further, consider a self-dual lattice $\cL\subset \RR^2$ (a rotation of the square lattice) given by
    \[\cL \, = \, \myfrac{1}{\sqrt{\phi+2}} \left\langle \begin{pmatrix}
        \phi \\ -1
    \end{pmatrix}, \begin{pmatrix}
        1 \\ \phi
    \end{pmatrix} \right\rangle^{}_{\ZZ}  \ =\, 
    \cL^{\circ} \, .
    \]
Here, $\phi$ stands for the golden mean.
We define a positive pure-point measure $\omega^{}_{\cL,\ts h}$ on $\RR$ via 
\[
\omega^{}_{\cL,\ts h}(g) \, = \, \delta^{}_{\cL}( g \otimes h) \, = \, \sum_{x\in\ZZ[\phi]} \ee^{-\tfrac{\pi (x^\star)^2}{3-\tau}} \, g \bigl(\tfrac{x}{\sqrt{\phi+2}}\bigr)
\]
for all $g\in\Cc(\RR)$, where $(m+n\phi)^{\star} = m+n(1-\phi)$. 
It has Dirac deltas on the countable dense set $\tfrac{1}{\sqrt{\phi+2}}\ZZ[\phi]$.
By standard arguments, the measure $\omega^{}_{\cL,\ts h}$ is a translation-bounded Poisson measure. In fact, it is a Fourier eigenmeasure, as 
\[
\widehat{\omega^{}_{\cL,\ts h}} \, = \, \omega^{}_{\cL^{\circ},\ts \widecheck{h}}
\, = \, \omega^{}_{\cL^{},\ts h^{}}\ .
\]
For details of the above construction, see~\cite{BSS23}.
\exend
\end{example}

We next discuss  Guinand's measure. This measure on the line was constructed by Meyer~\cite{Mey16} from a similar tempered distribution studied by Guinand \cite{Guinand}. Meyer showed that Guinand's measure is a crystalline measure that is strongly almost periodic as a tempered distribution~\cite{Mey17}. In fact, Meyer's results hold for a class of Guinand-like crystalline measures on the line. 

\begin{example}[Guinand's measure $\tau$]
Consider the measure
\[
\tau \, = \, \sum_{n=1}^{\infty} \chi(n) \ts \myfrac{r^{}_{3}(n)}{\sqrt{n}} \Big( \delta^{}_{\frac{\sqrt{n}}{2}} \ts - \ts  \delta^{}_{-\frac{\sqrt{n}}{2}}\Big)\ ,
\]
where the representation number $r_3(n)$ counts the number of different ways $n$ can be expressed as
a sum of three squares, and where $\chi(n)$ is defined by
\[\chi(n) \, = \, \left\{\begin{array}{rl}
    -\frac{1}{2}, & \mbox{if} \ n\in \NN \setminus 4\NN,\\
    4, & \mbox{if} \ n\in 4\NN \setminus 16\NN,\\
    0,  & \mbox{if} \ n\in 16\NN.\\
\end{array} \right.
\]
It follows from Guinand's work \cite{Guinand} that $\tau$ is a tempered distribution satisfying $\widehat{\tau} = -\ii \tau$.
\exend
\end{example}

\noindent 

The criterion based on $p$-discreteness (Lemma~\ref{lem:FQcrit}) is quite handy in this context.
Since the support $\vL$ of the tempered measure $\tau$ satisfies
\[ 
\vL \, \subset \, \Bigl\{\pm \myfrac{\sqrt{n}}{2} \ : \ n\in \NN \Bigr\}\ ,  
\]
we have for two neighbouring points 
\[ 
\Bigl| \myfrac{\sqrt{n+1}}{2} - \myfrac{\sqrt{n}}{n} \Bigr| \, \approx \, \myfrac{1}{4\sqrt{n}}\ .
\]
This shows that $\vL$ is $p$-discrete. We thus arrive at the following result, compare \cite[Thm.~4.3]{Mey17}.
\begin{proposition}
\label{prop:GMFQC}
    Guinand's measure is a Fourier quasicrystal. In particular, it is a strongly almost periodic tempered distribution.
\end{proposition}

\begin{proof}
    We have seen above that the support of Guinand's measure $\tau$ is $p$-discrete. Since $\hat{\tau} = -\ii \tau$, we have $\supp(\tau) = \supp(\widehat{\tau})$. Hence $\supp(\widehat{\tau})$ is also $p$-discrete. With Lemma~\ref{lem:Fav_p-discrete}, it follows that $\tau$ is Fourier quasicrystal. Strong almost periodicity follows from Lemma~\ref{lem:STimpliesTB}.
\end{proof}

\smallskip

We now discuss periodic Fourier eigenmeasures, which will be used for constructing crystalline measures outside the class of Fourier quasicrystals.
The idea of lattice-supported measures whose Fourier transform is also lattice-supported (but the relation does \emph{not} come from the Poisson summation formula) appears in \cite{K16}. Lattice-periodic crystalline measures that are Fourier eigenmeasures were studied in~\cite{BSS23}. 
These may have a gap in their support of arbitrary size at the origin, which makes them versatile tools for constructing non-periodic crystalline eigenmeasures. For later use, we include \cite[Cor.~7.4]{BSS23} with $n=m^2$.

\begin{lemma}[periodic eigenmeasures {\cite[Cor.~7.4]{BSS23}}]
\label{lem:eigenmesures}
Fix an eigenvalue $\lambda \in \{ \pm 1, \pm \ii\}$.
For each $m \geq 4$, there exists a non-zero measure $\mu^{}_m$ on the line with the following properties:
\begin{itemize}
    \item[(i)] $\supp (\mu^{}_m) \subset \frac{1}{m} \ZZ \subset \QQ$,
    \item[(ii)] $\mu^{}_m$ is $m \ZZ$-periodic,
    \item[(iii)] $\mu^{}_m$ is a Fourier eigenmeasure for eigenvalue $\lambda$, that is $\widehat{\mu^{}_m}=\lambda \ts \mu^{}_m $,
    \item[(iv)] $\mu^{}_m$ vanishes around the origin. More precisely, $|\mu^{}_{m}|\bigl( (-\frac{m}{4}+1,\frac{m}{4}-1)\bigr) =0$,
    \item[(v)] For all $x\in \supp(\mu^{}_{m})$, $|\mu^{}_{m}|(\{x\}) \leq 1$, and there exists some $x\in \RR$ such that $\mu^{}_{m}(\{x\}) =1$.  \qed
\end{itemize}
\end{lemma}

\begin{remark}[properly nested eigenmeasures]
\label{rem:choice_y} 
We will use the above Fourier eigenmeasures as building blocks for aperiodic crystalline measures.
\begin{itemize}
\item[(i)]  In this paper, we will \emph{always} work with sequences $(\sigma^{}_n)^{}_{n\in \NN}$ of such eigenmeasures with $\sigma^{}_n=\mu^{}_{k^{}_n}$, such that a central fundamental domain of $\sigma^{}_{n}$ is properly contained in the central gap of~$\sigma^{}_{n+1}$. We impose the condition $k^{}_{n+1}/4-1\geq k^{}_n/2+1$ with $k^{}_1\geq 4$ which is, for example, satisfied for the choice $k^{}_n=4^n$.  

\item[(ii)] This setting gives an abundance of points in $\supp(\sigma^{}_n)\setminus \bigcup_{m>n}\supp(\sigma^{}_{m})$, which will be used for further analysis. For example, there exist distinct numbers $x^{}_1,\, x'_1,\, x^{}_2,\, x'_2, \ldots$ such that $x^{}_n,\, x'_n\in \supp(\sigma^{}_n)\setminus \bigcup_{m>n}\supp(\sigma^{}_{m})$ and $|x^{}_n-x'_n|\leq 1/k^{}_n$ for all $n\in \NN$. Also, there exist distinct numbers $y^{}_1,\, y^{}_2, \ldots$ such that for every $n\in \NN$ we have $\sigma^{}_n(\{y^{}_n\})=1$ and $\sigma^{}_m(\{y^{}_n+y\})=0$ for all $m>n$ and all $y\in [-1,1)$. 

\item[(iii)] Given a sequence of such measures, and weights $a^{}_n \in \CC$, then
\[
\sigma\,=\,  \sum_{n\in \NN} a^{}_n \sigma^{}_n
\]
is a measure having locally finite support. If the sequence $(a^{}_n)^{}_{n\in \NN}$ grows at most polynomially in $n$ (in fact, as $\sigma^{}_n=0$ on $(-\frac{k^{}_n}{4}, \frac{k^{}_n}{4})$, we only need $a^{}_n$ to grow polynomially in $k^{}_n$), then $\sigma$ is a~tempered measure, which is a Fourier eigenmeasure to eigenvalue $\lambda$, see \cite[Thm.~7.5]{BSS23}. 

\item[(iv)] More generally, take a sequence $(\varOmega^{}_n)^{}_{n\in \NN}$ of pure-point probability measures with finite support inside a common compact set~$K$ and consider the measure
\[
\sigma\, =\,  \sum_{n\in \NN} a^{}_n (\varOmega^{}_n \ast \sigma^{}_n) \ ,
\]
compare \cite[Sec.~8]{BSS23}.
Then $\sigma$ is a measure having locally finite support. If the sequence $(a^{}_n)^{}_{n\in \NN}$ grows at most polynomially in $n$ (as above, we can, in fact, only require $a^{}_n$ to grow polynomially in $k^{}_n$), then $\sigma$ is a~tempered measure. \exend
\end{itemize}
\end{remark}

The construction in Remark~\ref{rem:choice_y} {\rm (iii)} allows for coefficient sequences that are of polynomial growth. In Section \ref{sec:sapdistr}, we will implement a modification due to Favorov in our setting, which allows us to work with coefficient sequences of superpolynomial growth. 
A variant of the construction in Remark~\ref{rem:choice_y} {\rm (iv)} will be used to obtain our example in Section~\ref{sec:NFQC2}. 

\smallskip

We close this section with a refinement on  Kolountzakis' construction \cite{K16}, which provides examples of aperiodic translation-bounded crystalline measures that are Fourier eigenmeasures. In fact, these measures $\mu$ are \emph{norm-almost periodic}, i.e., for all $\varepsilon>0$, the set
\[P^{B}_{\varepsilon}(\mu) \, := \, \{t\in \RR^d \, : \, \|\mu - \tau^{}_t \mu \|^{}_{B} <\varepsilon \} \]
of its $\varepsilon$-almost periods is relatively dense in $\RR^d$. (Recall that $A\subset \RR^d$ is \emph{relatively dense} if there is a compact set $K\subset \RR^d$ such that $A+K=\RR^d$.) In the above definition,  $\|\cdot\|^{}_B$ is the $B$-norm from Lemma~\ref{lem:tbchar}. Any of the equivalent $B$-norms turns the vector space of translation-bounded measures into a Banach space. Norm almost periodicity does not depend on the choice of $B$, and it implies strong almost periodicity as a measure, see \cite[Sec.~5.3]{MoSt} for details.

\begin{proposition}[an aperiodic translation-bounded crystalline Fourier eigenmeasure]
\label{prop:countreex_simple}
Take $(\sigma^{}_n)^{}_{n\in \NN}$  and $k_n^{}$ as in Remark~\ref{rem:choice_y} and consider the measure
\[
\sigma\,=\, \sum_{n\in \NN} \myfrac{1}{n^2 k^{}_n}\ts \sigma^{}_n \ .
\]
Then $\sigma$ is a crystalline Fourier eigenmeasure that is non-periodic and translation-bounded as a measure. Moreover, $\sigma$ is a norm almost periodic measure. In particular, $\sigma$ is strongly almost periodic as a measure.
\end{proposition}

\begin{proof}
 Indeed, crystallinity and the eigenmeasure property follow from Remark~\ref{rem:choice_y} {\rm (iii)}, and non-periodicity follows from Remark~\ref{rem:choice_y} {\rm (ii)}, as points in the support of $\sigma$ get arbitrarily close. In order to show translation boundedness, we use the $(0,1)$-norm and estimate
\[
\|\sigma\|^{}_{(0,1)} \,\leq \,\sum_{n\in \NN} \myfrac{1}{n^2 k^{}_n} \|\sigma^{}_{n}\|^{}_{(0,1)} \, \leq \, \sum_{n\in \NN} \myfrac{1}{n^2 k^{}_n} k^{}_n \, = \, \sum_{n\in\NN} \myfrac{1}{n^2}\, <\, \infty \ .
\]
In order to show that $\sigma$ is norm almost periodic, consider for $N\in \NN$ the approximants
\[
\sigma^{}_N\,=\, \sum_{n=1}^N \myfrac{1}{n^2 k^{}_n}\ts \sigma^{}_n \ .
\]
For each $n\in \NN$, the measure $\sigma^{}_n$ is $k^{}_n\ZZ$ -periodic. It follows that $\sigma^{}_N$ is $c^{}_N\ZZ$ periodic, where $c^{}_N=\prod_{n=1}^N k^{}_n$, and hence norm almost periodic.
As $(\sigma^{}_N)^{}_{N\in \NN}$ converges to~$\sigma$ in the $(0,1)$-norm, also $\sigma$ is norm almost periodic (which follows from a $3\varepsilon$-argument for the $\varepsilon$-almost periods of $\sigma$, as in the case of functions).
In particular, $\sigma$ is a~strongly almost periodic measure \cite[Lem.~7]{BM04}.
\end{proof}

\section{A characterisation of strong almost periodicity}\label{sec:WMC}

\subsection{Strongly almost periodic distributions}

Consider a strongly almost periodic tempered distribution $T\in \mathsf{SAP}^{}_{\cS}(\RR^d)$. It has been noted by Meyer \cite[Thm.~3.8]{Mey18} that its Fourier transform $\widehat T$ coincides with the Dirac comb weighted by the Fourier--Bohr coefficients of $T$, if $\widehat T$ is a measure. Thus, requiring $\widehat T$ to be a measure is a natural assumption that results in a~transparent Fourier theory. This observation holds in more generality, as we now describe.

\begin{definition}[Fourier--Bohr coefficients of a strongly almost periodic distribution]
Let $T\in \mathsf{SAP}^{}_{\Ccinf}(\RR^d)$ be a strongly almost periodic distribution and take $\varphi\in \Ccinf(\RR^d)$ such that $\widehat \varphi(0)=1$. We then call the number
\[
M(T)\, :=\, M(T*\varphi)
\]
the \emph{mean} of $T$. For $k\in \RR^d$ and $\chi^{}_k(x)=e^{2\pi\ii k\cdot x}$, we call $\mathbf{a}^{}_k(T)=M(\overline{\chi_k^{}}\cdot T)$ the \emph{Fourier--Bohr coefficient} of $T$ at $k$. 
\end{definition}

Note that the mean of $T$ is indeed well-defined due to $(T*\varphi)*\psi=(T*\psi)*\varphi\in SAP(\RR^d)$ for $\varphi,\psi\in L^1(\RR^d)$, compare Remark~\ref{rem:SAPstable}. Also, the Fourier-Bohr coefficient is well-defined as $\chi^{}_kT$ is a strongly almost periodic distribution for any $k\in \RR^d$ Indeed, note that for all $\varphi \in \Ccinf(\RR^d)$ we have 
\[
(\chi^{}_kT)*\varphi(t)\, =\, (\chi^{}_kT)(\tau^{}_t \varphi^\dagger) \, = \, T(\chi^{}_k\tau^{}_t \varphi^\dagger) \, = \, \chi^{}_k(t) \ts \bigl(T*(\chi^{}_k \varphi)\bigr)(t) \in SAP(\RR^d) \,.
\]
The following result is now immediate from Lemma~\ref{CharSAP}.

\begin{lemma}[strongly almost periodic tempered distributions] 
\label{lem:FBTD}
Let $T\in  \mathsf{SAP}^{}_{\Ccinf}(\RR^d)$ be a~strongly almost periodic distribution, and consider its at most countable Fourier--Bohr spectrum $S=\{k\in \RR^d: \mathbf{a}_{k}^{}(T)\ne0 \}$.  Assume, in addition, that $T$ is a tempered distribution. Then, the following statements are equivalent.
\begin{itemize}
\item[(i)] The distributional Fourier transform $\widehat T\in \cS'(\RR^d)$ of $T$ is a measure.
\item[(ii)] For every compact $K\subset \RR^d$, the sum $\sum_{k\in S\cap K} |\mathbf{a}^{}_{k}(T)|$ is finite.
\end{itemize}
If any of the above statements holds, then the identity $\widehat T=\sum_{k\in S} \mathbf{a}^{}_{k}(T)\ts \delta^{}_{k}$ holds in the measure sense.
\qed
\end{lemma}

In the following section, we will characterise strong almost periodicity in Poisson measures in terms of translation boundedness.

\subsection{Strong almost periodicity and translation boundedness}

Recall that translation boundedness is a necessary condition for strong almost periodicity. We will argue that, for Poisson measures and for crystalline measures, the two properties are in fact equivalent.
We first give a corresponding characterisation for distributions. We split the proof into two lemmas, each of which shows one direction. The first lemma is a reformulation of Lemma~\ref{lem:FBTD}.

\begin{lemma}\label{lem:SAP1}
Consider a tempered measure $T$ such that $\widehat T$ is a strongly almost periodic distribution. Then $T$ is a pure-point measure.
\qed 
\end{lemma}

The proof of the other direction is similar to that of Lemma~\ref{lem:SAPchar}. Here, we work with an approximate identity of Schwartz functions with compactly supported Fourier transform. Such an approximate identity is constructed as in Remark~\ref{rem:appid}, now with $\varphi=|\widehat{\psi\ts}|^2=\widehat{\psi *\widetilde{\psi\ts} \ts}$, where $\psi\in \Ccinf(\RR^d)$ is chosen so that $\|\psi\|^{}_2=1$. 

\begin{lemma}\label{lem:SAP3}
Consider a tempered measure $T$ and assume that $\widehat T$ is a translation-bounded distribution. If $T$ is pure point, then $\widehat T$ is a strongly almost periodic distribution.
\end{lemma}
\begin{proof}
 Consider any $f\in \Ccinf(\RR^d)$ and note that $\widehat T *f\in \Cu(\RR^d)$ by assumption. Take an approximate identity $(\varphi^{}_n)^{}_{n\in \NN}$ of Schwartz functions in $\RR^d$ with compactly supported Fourier transform and define $f^{}_n=f* \varphi^{}_n$.  As $\widehat T*f^{}_n=(\widehat T*f)*\varphi^{}_n$, the sequence $(\widehat T*f^{}_n)^{}_{n\in\NN}$ converges to $\widehat T*f$ uniformly. Also note that $\widehat T *f^{}_n=\reallywidehat{T\cdot \widecheck{\varphi^{}_n}\cdot \widecheck{f}}\in SAP(\RR^d)$, as $T\cdot \widecheck{\varphi^{}_n}\cdot \widecheck{f}$ is a finite measure, see e.g.~\cite[Lem.~4.8.10]{MoSt}. Since $SAP(\RR^d)$ is closed with respect to the supremum norm, we conclude $\widehat T*f\in SAP(\RR^d)$.
\end{proof}

Combining the previous two results, we get the following characterisation of strong almost periodicity for distributions. 

\begin{proposition}[strong almost periodicity for distributions]
\label{prop:M1-theorem}
Let $T$ be a tempered measure. Then $\widehat T$ is a strongly almost periodic distribution if and only if $\widehat T$ is a~translation-bounded distribution and $T$ is pure point. \qed
\end{proposition}

We can now use Lemma~\ref{lem:SAPchar} to obtain the analogous characterisation of strong almost periodicity of tempered distributions.
\begin{proposition}[strong almost periodicity for tempered distributions]
\label{prop:M-theorem}
Let $T$ be a tempered measure. Then $\widehat T$ is a strongly almost periodic tempered distribution if and only if $\widehat T$ is a~translation-bounded tempered distribution and $T$ is pure point. \qed
\end{proposition}

We also have a corresponding characterisation of strong almost periodicity for measures. We argue as above, using the following lemma, whose proof is exactly as in Lemma~\ref{lem:SAPchar}. 

\begin{lemma}
For $\mu\in \cM(\RR^d)$, we have $\mu\in \mathsf{SAP}^{}_{\Cc}(\RR^d)$ if and only if \/$\mu\in \cM^\infty(\RR^d)$ and, for all $\varphi\in \Ccinf(\RR^d)$, we have $\mu*\varphi\in SAP(\RR^d)$. \qed
\end{lemma}

\begin{proposition}[strong almost periodicity for measures]\label{thm:SAPTBM}
Let $\mu$ be a tempered measure. Then $\widehat \mu$ is a strongly almost periodic measure if and only if $\widehat \mu$ is a translation-bounded measure and $\mu$ is pure point. \qed
\end{proposition}

For Poisson measures, we thus get the following characterisation of strong almost periodicity. Recall that every crystalline measure is a Poisson measure. Characterisation (c) for measures already appears in \cite[Lem.~5.8]{Mey18}.

\begin{theorem}[strong almost periodicity in Poisson measures]\label{thm:sappois}
For any Poisson measure $\mu$, the following hold true.
\begin{itemize}
\item[(a)] $\mu$ is a~strongly almost periodic distribution if and only if $\mu$ is a translation-bounded distribution. 
\item[(b)] $\mu$ is a~strongly almost periodic tempered distribution if and only if $\mu$ is a translation-bounded tempered distribution. 
\item[(c)] $\mu$ is a~strongly almost periodic measure if and only if $\mu$ is a translation-bounded measure. \qed
\end{itemize} 
\end{theorem}

Let us recall that Meyer's conjecture has already been proved for the class of Fourier quasicrystals \cite[Cor.~2]{Fav24b}. This can be seen as a consequence of Theorem~\ref{thm:sappois} (b), in view of Lemma~\ref{lem:temptb}. 
Theorem~\ref{thm:sappois} (b) can also be used to prove Meyer's almost periodicity conjecture for positive Poisson measures, compare Favorov's statement \cite[Thm.~1]{Fav24}. Indeed, positivity implies translation-boundedness as a measure by the following lemma, and translation-boundedness as a measure implies translation-boundedness as a tempered distribution by Lemma~\ref{lem:TBMimpTBTD}.

\begin{lemma}\label{lem:poswap}
Assume that $\mu$ is a positive tempered measure whose Fourier transform is a~measure. Then $\mu$ is a translation-bounded measure, and $\widehat \mu$ is a translation-bounded tempered distribution.
\end{lemma}

\begin{proof}
As $\mu$ is a tempered distribution, we have  $\widehat \mu(\varphi * \widetilde \varphi)=\mu(|\widehat \varphi|^2)\geq 0$ for any $\varphi\in \cS(\RR^d)$, where $\widetilde \varphi(x)=\overline{\varphi(-x)}$. Hence, $\widehat \mu$ is a positive definite tempered distribution. As $\Ccinf(\RR^d)$ is dense in $\Cc(\RR^d)$ with respect to the supremum norm, we can argue $\widehat \mu(\varphi * \widetilde \varphi)\geq 0$ for any $\varphi\in \Cc(\RR^d)$. Thus, the measure $\widehat \mu$ is positive definite as a measure. As a consequence, its Fourier transform $\mu$ is translation bounded \cite[Prop.~4.9]{BF75}. And as $\mu$ is strongly tempered (as it is positive and tempered), the statement about $\widehat\mu$ now follows from Lemma~\ref{lem:temptb}.
\end{proof}

\section{Counterexamples to Meyer's almost periodicity conjecture}

Favorov gave an example of a crystalline measure that is not strongly almost periodic as a tempered distribution. Based on Favorov's construction, we give another example of a crystalline measure that is not strongly almost periodic as a general distribution. Thus, within the class of crystalline measures, strong almost periodicity might be violated in its weakest form.

\subsection{Favorov's counterexample}\label{sec:FCE}
Reading the arguments leading to the counterexample to Meyer's conjecture \cite[Theorem]{Fav24c}, the author actually proves the following.

\begin{theorem}[Favorov's measure $\nu^{}_{\mathrm{Fav}}$]\label{thm-fav} 
There exists a measure $\nu^{}_{\mathrm{Fav}}\in\cM(\RR)$ with the following properties:
\begin{itemize}
    \item[(a)] $\nu^{}_{\mathrm{Fav}}$ is a tempered pure-point measure with locally finite support, which can be chosen so that $\supp(\nu^{}_{\mathrm{Fav}}) \cap \QQ = \varnothing$.
    \item[(b)] Its Fourier transform $\widehat{\nu^{}_{\mathrm{Fav}}}$ as a tempered distribution is a tempered pure-point measure with locally finite support, such that $\supp(\widehat{\nu^{}_{\mathrm{Fav}}}) \subset \QQ$.
    \item[(c)] $\widehat{\nu^{}_{\mathrm{Fav}}} \notin \cS'_{\infty}(\RR)$. In particular, $\nu^{}_{\mathrm{Fav}}$ is not strongly tempered. \qed
\end{itemize} 
\end{theorem}
Thus, Favorov provides an example of a tempered distribution $\widehat{\nu^{}_{\mathrm{Fav}}}\in \cS'(\RR)$ that is not translation bounded as a tempered distribution. In particular, he finds a Schwartz function $\varphi\in \cS(\RR)$ such that $\widehat{\nu^{}_{\mathrm{Fav}}}*\varphi$ is unbounded. 
In fact, Favorov's function $\varphi$ is not supported on a compact set. Its Fourier transform (see his Lemma 2) is an infinite sum of non-negative $\Ccinf(\RR)$-functions with mutually disjoint supports. Now, if $\varphi$ were compactly supported, its Fourier transform would be the restriction of an entire function. Thus, its zero set cannot contain a nonempty open interval, which is contradictory.

\smallskip

Whereas Favorov's measure $\nu^{}_{\mathrm{Fav}}$ is not strongly tempered, it is unclear whether it is translation unbounded as a tempered distribution. However, Favorov's measure leads to the following interesting $\frac{1}{2}$-example of a measure $\mu$ such that both $\mu$ and $\widehat \mu$ are translation unbounded as a tempered distribution.  We call it a half example because, in the proof, we provide two examples and show that one of them works, but we do not specify which. The underlying construction will later be used in our first example, see Theorem~\ref{tnm:ex1}.

\begin{theorem}[$\frac{1}{2}$-example]\label{thm-half} 
There exists a crystalline measure $\mu$, such that neither $\mu$ nor $\widehat{\mu}$ is translation bounded as a tempered distribution. In particular, neither $\mu$ nor $\widehat{\mu}$ are strongly tempered measures. 
\end{theorem}

\begin{proof}
In order to simplify the presentation of our argument, we always mean ``translation bounded as a tempered distribution" if we use the term ``translation bounded" in the following. 

\smallskip

\noindent We only need to show that neither $\mu$ nor $\widehat{\mu}$ is translation bounded. Once we do this, the fact that $\widehat \mu$ is not translation bounded implies that $\mu$ is not strongly tempered, by applying Lemma~\ref{lem:temptb}. On the other hand, the fact that $\mu$ is not translation bounded implies that $\widehat \mu$ is not strongly tempered.

\smallskip

\noindent Let $\nu^{}_{\mathrm{Fav}}$ be the measure from Theorem~\ref{thm-fav}. We split the proof into two cases.

\smallskip

\noindent\emph{Case 1:} $\nu^{}_{\mathrm{Fav}}$ is not translation bounded. Then $\mu:= \nu^{}_{\mathrm{Fav}}$ satisfies the required properties.

\smallskip

\noindent\emph{Case 2:} $\nu^{}_{\mathrm{Fav}}$ is translation bounded. Define 
\[
\mu\, :=\, \nu^{}_{\mathrm{Fav}}+\widehat{\nu^{}_{\mathrm{Fav}}} \,.
\]
Since $\nu^{}_{\mathrm{Fav}}$ is translation bounded and $\widehat{\nu^{}_{\mathrm{Fav}}}$ is not, neither is $\mu$, as follows from the estimate
\[
|(\mu \ast \varphi)(x)| \,= \, \bigl|(\nu^{}_{\mathrm{Fav}}\ast \varphi)(x)+(\widehat{\nu^{}_{\mathrm{Fav}}}\ast \varphi)(x)\bigr| \, \geq \, 
\bigl|(\widehat{\nu^{}_{\mathrm{Fav}}}\ast \varphi)(x)\bigr| - \bigl|(\nu^{}_{\mathrm{Fav}}\ast \varphi)(x)\bigr| \ .
\]
The same argument shows that
\[
\widehat{\mu}\, =\,  \widehat{\nu^{}_{\mathrm{Fav}}}+\nu^{\dagger}_{\mathrm{Fav}}
\]
is not translation bounded. 
\end{proof}

A key ingredient of Favorov's proof is the following. One picks a~sequence $(t^{}_{k})^{}_{k}$ such that $t^{}_{k} \longrightarrow 0$ as $k\to \infty$, and subsequently, when evaluating $\nu^{}_{\mathrm{Fav}}(\varphi)$, one has to deal with expressions of the form 
\[ \varphi(x-t^{}_n) - \varphi(x)\, ,\]
where $\varphi \in \cS(\RR)$. Since 
\[ \varphi(x-t^{}_n) - \varphi(x) \, = \,t^{}_{n} \, \myfrac{\varphi(x-t^{}_n) - \varphi(x)}{t^{}_{n}} \, ,\]
we have
\[ \myfrac{\varphi(x-t^{}_n) - \varphi(x)}{t^{}_{n}} \, \xrightarrow{\ n\to \infty\ }\,  D\varphi (x) \]
in the distributional sense (and here, $D$ denotes the derivative). As $\varphi' \in \cS(\RR)$, it has bounded derivative and one can choose $(t^{}_{n})^{}_{n\in \NN}$ converging fast enough to ensure the convergence. 
This, nevertheless, does not remain true once one moves outside of $\cS(\RR)$.

We will use this trick in a slightly different way as we construct our examples in the upcoming sections.

\subsection{A second counterexample}\label{sec:sapdistr}
Motivated by Favorov's arguments, we construct a~crystalline measure that is unbounded as a distribution.
Thus, Meyer's almost periodicity conjecture does not hold for the weakest form of almost periodicity at hand.  To do that, we will rely, again, on the existence of Fourier eigenmeasures with certain properties, as stated in Lemma~\ref{lem:eigenmesures} and with properties mentioned in Remark~\ref{rem:choice_y}. 

\smallskip

Our construction is described in Theorem~\ref{tnm:ex1}. We choose any compactly supported test function and define a crystalline measure such that its convolution with the chosen test function (multiplied by a particular character) is unbounded. We first construct a general family of crystalline measures (Proposition~\ref{cor1}). We will then use an auxiliary result (Lemma~\ref{lem-bohr}) concerning the Fourier--Bohr coefficients of Bohr almost periodic functions to derive the desired result.

\smallskip

Now, we establish a tool for constructing crystalline measures in Proposition~\ref{cor1}, which extends Favorov's construction. To show that this leads to tempered measures, we rely on the following estimate for the periodic pure-point measures from Lemma~\ref{lem:eigenmesures}.

\begin{lemma} 
\label{lem-tech}
Fix any $m\in\NN$.
Let $\mu$ be any pure-point measure on the line that is supported within $\frac{1}{m}\ZZ$ and that is  $m\ZZ$-periodic, with complex weights inside the unit disk. 
Let $f \in C^1(\RR)$ be any function such that 
\[
    \bigl\| (1+x^2) f'(x) \bigr\|^{}_{\infty} < \infty \ .
\]
We then have for any $t\in (0,1)$ the estimate
\[
\bigl| \ts (\tau^{}_{t}\mu-\mu)(f) \bigr| \, \leq \,  \, t \cdot (4m^2+m) \cdot \bigl\| (1+x^2) f'(x) \bigr\|^{}_\infty \ .
\]
\end{lemma}
\begin{proof}
Write $\mu$ in explicit form as  
\[
\mu\, = \, \sum_{n \in m\ZZ} \sum_{k=0}^{m^{2}-1} c^{}_k \delta^{}_{\frac{k}{m}+n} \,,
\]
with complex weights $|c^{}_k|\leq 1$. We now estimate
\begin{align*}
\left| (\tau^{}_{t}\mu-\mu)(f) \right| &\, =\,  \left| t \, \mu \Bigl(\myfrac{\tau^{}_{-t}f -f}{t}\Bigr)\right| 
\, \leq \,  t \cdot \sum_{n \in m \ZZ} \ \sum_{k=0}^{ m^{2}-1} |c^{}_k| \left| \myfrac{\tau^{}_{-t}f\bigl(\frac{k}{m}+n\bigr) -f\bigl(\frac{k}{m}+n\bigr)}{t} \right| \\[3pt]
&\, \leq \,  t \cdot\sum_{n \in m \ZZ} \ \sum_{k=0}^{m^{2}-1}  \left| \myfrac{f\bigl(\frac{k}{m}+n+t\bigr) -f\bigl(\frac{k}{m}+n\bigr)}{t} \right|
\,=  \,  t \cdot\sum_{\ell \in \frac{1}{m} \ZZ}  \left| \myfrac{f(\ell+t) -f(\ell)}{t} \right| \, .
\end{align*}
By the mean value theorem, for each $\ell \in \frac{1}{m}\ZZ$, there exists some $\xi^{}_{\ell}$ such that
$\ell <  \xi^{}_{\ell} < \ell+t$ and 
\[
\myfrac{f(\ell+t) -f(\ell)}{t}  \, = \, f'(\xi^{}_{\ell})\,.
\]
This leads to the estimate
\[
\left| (\tau^{}_{t}\mu-\mu)(f) \right| \, \leq \,  t \cdot\sum_{\ell \in \frac{1}{m} \ZZ}  \left| f'(\xi^{}_{\ell}) \right| \, \leq \,   t \cdot \ts \bigl\| (1+x^2) f' \bigl\|^{}_\infty \,  \sum_{\ell \in \frac{1}{m} \ZZ} \myfrac{1}{1+|\xi^{}_{\ell}|^2} \, . 
\]
Since for all $\ell \geq 0$, we have $1+|\xi^{}_{\ell}|^2 \geq 1+\ell^2$, while for $\ell<0$, we have $1+|\xi^{}_{\ell}|^2 \geq 1+\left( \ell+1 \right)^2$, we can further estimate
\begin{align*}
 \sum_{\ell \in \frac{1}{m} \ZZ} \myfrac{1}{1+|\xi^{}_{\ell}|^2}& \,\leq\,   \sum_{\substack{\ell \in \frac{1}{m} \ZZ \\ \ell \geq 0} } \myfrac{1}{1+\ell^2} + \sum_{\substack{\ell \in \frac{1}{m} \ZZ \\
 \ell < 0} } \myfrac{1}{1+(\ell+1)^2}\, = \, \sum_{\ell \in \frac{1}{m} \ZZ } \myfrac{1}{1+\ell^2}+\sum_{\substack{j \in \frac{1}{m} \ZZ \\
 0 \leq j < 1} } \myfrac{1}{1+j^2} \\[5pt]
&\, \leq\,  m^2 \Bigl(\sum_{\ell\in \ZZ}  \myfrac{1}{ m^2+\ell^2} \Bigl)+m \,\leq \, m^2 \Bigl(\sum_{\ell\in \ZZ} \myfrac{1}{ 1+\ell^2} \Bigl)+m \ ,
\end{align*}
where
\[
\sum_{\ell\in \ZZ}  \frac{1}{ 1+\ell^2} \,=\,\pi\coth{\pi} \,=\, 3.153\ldots \,\leq\, 4 \,.
\]
This yields the claim of the lemma.
\end{proof}

\begin{proposition}[constructing crystalline measures]
\label{cor1} 
Let $(k^{}_n)^{}_{n\in \NN}$ be any increasing sequence of positive integers starting with $k^{}_1\geq 4$. Take the periodic Fourier eigenmeasures $(\mu^{}_m)^{}_{m\in \NN}$ to eigenvalue $\lambda$ from Lemma~\ref{lem:eigenmesures} and define  $\sigma^{}_{n}=\mu^{}_{k^{}_n}$ for each $n\in \NN$. Consider any sequences  $(a^{}_n)^{}_{n\in \NN}$ of real numbers and  $(t^{}_{n})^{}_{n\in \NN}$ in $(0,1)$ such that
\begin{equation}
\label{eq-tech-cond}
\sum_{n=1}^\infty |a^{}_n| \cdot t^{}_n\cdot  (4k^{2}_n+k^{}_n)  < \infty \ .
\end{equation}
Now define
\[
\sigma\, :=\,  \sum_{n=1}^\infty a^{}_{n}(\tau^{}_{t^{}_n}\sigma^{}_n-
\sigma^{}_n) \, \qquad \mbox{and} \qquad
\Sigma\, :=\,  \lambda \sum_{n=1}^\infty a^{}_{n} \bigl(\ee^{2 \pi \ii x \cdot t^{}_n } -1\bigr) \ts\sigma^{}_n \ .
\]
Then, both $\sigma$ and $\Sigma$ are tempered pure-point measures, with $\widehat{\sigma} = \Sigma$ as tempered distributions. Moreover, their support is locally finite and satisfies
\[
\supp(\sigma) \,\subset\, \QQ \ts \cup \big(\ts \bigcup_{n\in\NN}( -t^{}_n+\QQ)\big)\qquad \mbox{and} \qquad  \supp(\Sigma) \, \subset \, \QQ \,.
\]
\end{proposition}

\begin{remark}
\label{rem:stchoice} 
For example, with $ k^{}_n=2^{2n}$,  any sequences $(a^{}_n)^{}_{n\in \NN}$ and  $(t^{}_n)^{}_{n\in \NN}$ such that 
\[
\sum_{n=1}^\infty |a^{}_{n}| \cdot t^{}_{n} \cdot 2^{4n} \, < \, \infty 
\quad \text{ and } \quad 
  0\, <\,  t^{}_n \,<\, \myfrac{1}{2^{3n^2}} 
\]
satisfy the conditions of the proposition. In our construction below, we will need to select $(a^{}_n)^{}_{n\in \NN}$ of possibly fast growth. We can then pick $(t^{}_n)^{}_{n\in\NN}$ sufficiently small and satisfying extra restrictions, such as being irrational and being linearly independent over $\QQ$. \exend
\end{remark}
\begin{proof}[Proof of Proposition~\ref{cor1}] 
With the notation as above, consider the approximants
\[
\varrho_{N}\, := \, \sum_{n=1}^N a^{}_{n}(\tau^{}_{t^{}_n}\sigma^{}_n-
\sigma^{}_n) \,.
\]
We will prove the following statements, which then give the desired claim.
\begin{itemize}
\item[(a)] The sequence $(\varrho^{}_N)^{}_{N\in \NN}$ converges, in the vague topology, to the pure-point measure~$\sigma$, which has locally finite support.
\item[(b)] The sequence $(\widehat{\varrho^{}_N})^{}_{N\in \NN}$  converges, in the vague topology, to the pure-point measure~$\Sigma$, which has locally finite support.
\item[(c)] The sequence $(\varrho^{}_N)^{}_{N\in \NN}$ converges, in the topology of tempered distributions, to some $T \in \cS'(\RR)$.
\item[(d)] For all $\varphi \in \Ccinf(\RR)$ we have $\sigma(\varphi) \, = \, T(\varphi)$ and $\Sigma(\varphi) \, = \, \widehat{T}(\varphi)$. 
That is, $\sigma$ and $\Sigma$ are tempered measures satisfying $\cF(\sigma)= \Sigma$.  
Here, $\mathcal{F}$ denotes the Fourier transform in the tempered distribution sense.
\end{itemize}
\textbf{Ad (a):} Let $R>0$ be arbitrary. 
By construction, there exists some $N^{}_0\in\NN$ such that
\[
\tau^{}_{t^{}_n}\sigma_n-\sigma_n \, \equiv\,  0 \mbox{ on } [-R,R\ts] 
\]
for all $n \geq N^{}_0$. Let $f \in \Cc(\RR)$ supported within $[-R,R\ts ]$. We then have 
$\varrho^{}_N(f)=\varrho^{}_{N^{}_0}(f)$ for all $N\geq N^{}_0$.
This implies that the sequence $(\varrho^{}_N)^{}_{N\in \NN}$ is Cauchy in the vague topology and hence convergent to a measure, which must coincide with $\sigma$. Moreover, for each $R>0$, there exists some $N\in \NN$ such that 
\[
\sigma|^{}_{[-R,R\ts ]}\, =\,  (\varrho^{}_N)|^{}_{[-R,R\ts ]}
\]
is a pure-point measure with finite support, so claim (a) follows. Note here that if 
\[
\sigma \, = \, \sum_{n=1}^\infty a^{}_{n}(\tau^{}_{t^{}_n}\sigma^{}_n-
\sigma^{}_n)
\]
is evaluated on $f \in \Cc(\RR)$, at each point, only finitely many summands are non-zero.

\noindent \textbf{Ad (b):} For each $N\in \NN$, we have 
\[
\reallywidehat{\varrho^{}_N}\, = \,\sum_{n=1}^N a^{}_{n}\, \reallywidehat{(\tau^{}_{t^{}_n}\sigma^{}_n -
\sigma^{}_n)} \, = \, \sum_{n=1}^N a^{}_{n} \bigl(\ts \ee^{2 \pi \ii x \cdot t^{}_n } -1\bigr) \,\widehat{\sigma^{}_n}, = \, \lambda \sum_{n=1}^N a^{}_{n} \bigl(\ts \ee^{2 \pi \ii x\cdot t^{}_n } -1\bigr) \ts \sigma^{}_n \,.
\]
Exactly as in (a), since $\sigma^{}_n$ is pure point with uniformly discrete support and has a gap around~$0$, which is growing to infinity with $n$, we get that $\reallywidehat{\varrho^{}_N}$ converges in the vague topology to the pure-point measure 
\[
\Sigma \, = \, \lambda \sum_{n=1}^\infty a^{}_{n} \bigl( \ts \ee^{2 \pi \ii x\cdot  t^{}_n} -1\bigr) \ts \sigma^{}_n \ ,
\]
which has locally finite support. 

\noindent \textbf{Ad (c):} Note that, by Lemma~\ref{lem-tech}, we have for all $f \in \cS(\RR)$ and all $n\in \NN$ the estimate
\[
\left| \bigl(a^{}_{n}(\tau^{}_{t^{}_n}\sigma^{}_n-
\sigma^{}_n) \bigr) (f) \right| \,\leq\, |a^{}_n| \cdot t^{}_n \cdot (4k^{2}_n+k^{}_n) \cdot \bigl\| (1+x^2) f' \bigr\|^{}_\infty  \ .
\]
Equation \eqref{eq-tech-cond} then implies that $\bigl(\varrho^{}_{N}(f)\bigr)^{}_{N\in \NN}$ is a Cauchy sequence for arbitrary function $f \in \cS(\RR)$. Therefore, the sequence $(\varrho^{}_N)^{}_{N\in \NN}$ is Cauchy in the tempered distribution topology and hence converges to some $T \in \cS'(\RR)$. This proves~(c).

\noindent \textbf{Ad (d):} In part (a), we showed that, for all $\varphi \in \Ccinf(\RR)$, one has $\sigma(\varphi)=\lim_{N \to \infty} \varrho^{}_N(\varphi)$. In part (c), we also showed that $T(\varphi)=\lim_{N \to \infty} \varrho^{}_N(\varphi)$. 
Therefore, we have $\sigma(\varphi)= T(\varphi) $ for all $\varphi\in \Ccinf(\RR^d)$.
Next, since the sequence $(\varrho^{}_N)^{}_{N\in\NN}$ converges to $T$ in the topology of tempered distributions, we also have that the sequence $(\widehat{\varrho^{}_N})^{}_{N\in \NN}$ converges to $\widehat T$. 
Thus, for all $\varphi \in \Ccinf(\RR)$ we also have 
\[
\Sigma(\varphi)\, =\, \lim_{N\to \infty} \widehat{\varrho^{}_N}(\varphi)\, =\,  \widehat{T}(\varphi) \,.
\]
This proves (d) and finishes the proof.
\end{proof}

The above construction can now be used to yield tempered measures $\sigma$ that are not strongly tempered. 

\begin{lemma}[crystalline measures that are not strongly tempered]
\label{Lem3.16} 
Take sequences $(k^{}_n)^{}_{n\in \NN}$ and $(y_n^{})^{}_{n\in \NN}$ satisfying the assumptions in Remark~\ref{rem:choice_y}. Take sequences $(a^{}_n)^{}_{n\in \NN}$ and $ (t^{}_n)^{}_{n\in \NN}$ satisfy the assumptions in Proposition~\ref{cor1}. 
We then have
\[
\sigma\bigl(\{y^{}_n-t^{}_n \}\bigr)  \, \geq \,  a^{}_n- \sum_{k=1}^{n-1} \left| a^{}_k \right|  \,.
\]
\noindent In particular, if $a^{}_n$ satisfies for all $n\in \NN$ the estimate
\[
|a^{}_n| \,  \geq \,   \sum_{k=1}^{n-1} \left| a^{}_k \right| + 2^{k^{}_n}\, ,
\]
then $\sigma$ is not strongly tempered. 
\end{lemma}

\begin{proof}
By assumption, we have $\sigma^{}_m(\{y_n+y\})=0$ for all $m>n$ and all $y\in [-1,1)$.
Therefore, 
\begin{align*}
\left| \sigma\bigl(\{y^{}_n-t^{}_n\}\bigr) \right| &\, = \, \Bigl| \sum_{k=1}^n a^{}_k \tau^{}_{t^{}_k} \sigma^{}_k \bigl(\{y^{}_{n}-t^{}_n\} \bigr) \Bigr| \\
& \, \geq \,  \left|a^{}_n \tau^{}_{t^{}_n} \sigma^{}_n \bigl(\{y^{}_n-t^{}_n\}\bigr) \right| - \Bigl| \sum_{k=1}^{n-1} a^{}_k \tau^{}_{t^{}_k} \sigma^{}_k \bigl(\{y^{}_n-t^{}_n\}\bigr) \Bigr| \, \geq \, |a^{}_n|- \sum_{k=1}^{n-1} \left| a^{}_k \right| \ ,
\end{align*}
which proves the first claim. Next, let us note that, under the extra assumption, we obtain
\[
\pushQED{\qed}
\left| \sigma\bigl(\{y^{}_n-t^{}_n\}\bigr) \right| \, \geq \, |a^{}_n|- \sum_{k=1}^{n-1} \left| a^{}_k \right| \, \geq \, 2^{k^{}_n}\ .
\]
As $|\sigma|\bigl(B^{}_{k^{}_n}(0)\bigr)\geq 2^{k^{}_n}$ for all $n\in \NN$, it follows that $\sigma$ is not strongly tempered.
\end{proof}

\begin{remark} Under the conditions of Lemma~\ref{Lem3.16}, if we further assume that all $t^{}_n$ are linearly independent over $\QQ$, we get for all $m\in \NN$ that
\[
\tau^{}_{t^{}_m}\sigma^{}_m \bigl(\{y^{}_n-t^{}_n \}\bigr) \, = \, 0 \,.
\]
Therefore, we get the sharper inequality $\bigl|\sigma\bigl(\{y^{}_n-t^{}_n \}\bigr)\bigr| \, \geq \,  |a^{}_n|$.
\exend
\end{remark}

We will need the following refined result about Fourier--Bohr coefficients later on. Note that we do not assume $T \in \mathsf{SAP}^{}_{\cS}(\RR^d)$, not even $T\in \cS'_{\infty}(\RR^d)$. We only assume that $T*f$ is Bohr almost periodic for \emph{some} fixed Schwartz function $f$.

\begin{lemma}
\label{lem-bohr} 
Let $T \in \cS'(\RR^d)$ be such that $\widehat{T}$ agrees with a tempered measure $\mu$. Assume that $T*f \in SAP(\RR^d)$ for some $f \in \cS(\RR^d)$.
Then, $\widehat f\cdot \mu$ is a pure-point measure, and the Fourier--Bohr coefficients $\mathbf{a}^{}_{k}(T*f)$ satisfy
\[
\mathbf{a}^{}_{k}(T*f)\, =\, \widehat{f}(k) \mu\bigl(\{k \}\bigr) 
\]
for all $k\in\RR^d$. In particular, we have
\[
\sum_{k \in \RR^d} \Bigl|\ts \widehat{f}(k) \mu\bigl(\{k \}\bigr)\Bigr|^2 \, < \, \infty \,.
\]
\end{lemma}

\begin{proof}
Consider any $\varphi \in \Ccinf(\RR^d)$ and note that $\nu:=\, \widehat{f} \cdot \varphi \cdot \mu$
is a finite measure on $\RR^d$. Since we have $\nu =  \widehat{f} \cdot \varphi \cdot \widehat{T}$ as a tempered distribution, we obtain for its inverse Fourier transform 
\[
\widecheck{\nu} \,= \, T *f *\widecheck{\varphi} \,.
\]
Since $T*f \in SAP(\RR^d)$ and $\widecheck{\varphi} \in \cS(\RR^d) \subset L^1(\RR^d)$, we have $T *f *\widecheck{\varphi} \in SAP(\RR^d)$ by Remark~\ref{rem:SAPstable}. Thus $\nu$ is a finite pure point measure by Lemma~\ref{CharSAP}, and we have  
\[
\varphi(k) \cdot \mathbf{a}^{}_{k}(T *f  )\, =\,\mathbf{a}^{}_{k}(T *f *\widecheck{\varphi} )\, =\, \mathbf{a}^{}_{k}(\widecheck{\nu} ) \, =\,  \nu(\{k\}) =\varphi(k) \cdot  \widehat{f}(k) \cdot \mu(\{k\}) \,.
\]
Since $\widehat{f} \cdot \varphi \cdot \mu$ is a finite pure point measure for all $\varphi \in \Ccinf(\RR^d)$, we get that $\widehat{f} \cdot \mu$ is a pure-point measure, and the formula for the Fourier--Bohr coefficients follows.
Square summability follows from Bessel's inequality applied to the Bohr almost periodic function $T*f$. 
\end{proof}

\begin{cor}
    Under the conditions of Lemma~Lemma~\ref{lem-bohr}, if the function $f$ satisfies that the set $\{\xi \in \RR^d \, : \, \widehat{f}(\xi) = 0 \}$ is at most countable, then the measure $\mu$ is pure point.
\end{cor}
\begin{proof}
    Since $\widehat{f} \mu$ is pure point, the continuous part $(\widehat{f} \mu)^{}_{\mathsf{cont}}$ of $\widehat{f} \mu$ vanishes. Since we have $(\widehat{f} \mu)^{}_{\mathsf{cont}} = \widehat{f} \mu^{}_{\mathsf{cont}}$, it follows that $\mu^{}_{\mathsf{cont}}$ is supported inside $\{\xi \in \RR^d \, : \, \widehat{f}(\xi) = 0 \}$. As this set is countable, we obtain that $\mu$ is pure point. 
\end{proof}

\begin{remark}
    Note that if $T\ast f \in SAP(\RR^d)$ for a particular test function $f\in \cS(\RR^d)$, one cannot, in general, conclude anything about the strong almost periodicity of $T$. To see that,  pick $T$ to be a twice Fourier transformable measure such that $\widehat{T}$ is \emph{not} pure point, but its restriction to $B^{}_R(0)$ is pure point. Then $T*f$ is Bohr almost periodic function for all $f \in \cS(\RR^d)$ with $\supp(\widehat{f}) \subset B^{}_R(0)$. 
    
    \noindent More generally, despite the fact that $\cS \ast \cS = \cS$, we always have $f\ast \cS \subsetneq \cS$ \cite{PV78,Voi84}, and in the best-case scenario, one obtains a dense subspace. This happens if and only if $\{\xi \in \RR^d \, : \, \widehat{f}(\xi) = 0 \} = \varnothing$, as follows by the fact that Fourier transform is a~bijection on $\cS$ together with standard arguments using the geometric version of Hahn--Banach theorem \cite[Thm.~4.7]{Rudin}. Therefore, if one has a translation-bounded tempered distribution and test it with any such function (say a Gaussian), one can then conclude its almost periodicity. \exend
\end{remark}

Based on the construction described in Remark~\ref{rem:choice_y},
we will now give examples of crystalline measures $\sigma$ on $\RR$ that are not slowly increasing (hence not strongly tempered) and, foremost, if convolved with a~$\Ccinf(\RR)$-test function (which we also specify), one obtains a~function which is not Bohr almost periodic.

\begin{theorem}\label{tnm:ex1}
Let $\varphi \in \Ccinf(\RR)$ be an arbitrary function that is not identically vanishing. Take eigenmeasures $(\sigma^{}_n)^{}_{n\in \NN}$ and rational numbers $(y^{}_n)^{}_{n\in \NN}$ as in Remark~\ref{rem:choice_y}.
Then, we can choose a real number $y$, a decreasing sequence $(t^{}_{n})^{}_{n\in \NN}$ of positive numbers converging to zero and an increasing sequence $(a^{}_{n})^{}_{n\in \NN}$ of positive numbers diverging to $\infty$,
such that 
\[
\sigma \ := \, \sum_{k=1}^\infty a^{}_k(\tau^{}_{t^{}_k}\sigma^{}_k - \sigma_{k}) \qquad \mbox{and} \qquad 
\Sigma\, :=\,  \lambda \sum_{k=1}^\infty a^{}_k \bigl(\ee^{2 \pi \ii x \ts \cdot t^{}_k}-1\bigr) \ts \sigma^{}_k 
\]
are crystalline measures with $\widehat{\sigma}\, =\, \Sigma$, which satisfy
\[
\sum_{n=1}^\infty \bigl| \,  \sigma\bigl(\{y^{}_{n}-t^{}_n\}\bigr) \ts\widehat{\varphi}(y^{}_n-t^{}_{n}-y)\bigl|^2\, =\, \infty \,.
\]
Moreover, they satisfy
\[
\Sigma\bigl(\{ y^{}_{n}-t^{}_n \}\bigr)\,  =\,  \sigma\bigl(\{ -y^{}_{n}+t^{}_n \}\bigr) \, =\, \Sigma\bigl(\{ -y^{}_{n}+t^{}_n \}\bigr) \, =\, 0  
\]
for all $n\in \NN$. As a consequence, 
\[
\eta\, :=\, \sigma + \cF(\sigma)+\cF^2(\sigma) + \cF^3(\sigma)
\]
is a crystalline Fourier eigenmeasure satisfying $\widehat\eta=\lambda\eta$, and with $\psi(x)=\ee^{2 \pi \ii x\cdot y} \varphi(x)$ we have
\[
\eta\ast \psi \, \notin \, SAP(\RR) \,.
\]
Thus, $\eta$ fails to be strongly almost periodic or translation bounded as a distribution, and $\eta$~fails to be a strongly tempered measure.
\end{theorem}
\begin{remark}
The following argument works for any $\varphi \in \cS(\RR)$ that is not identically vanishing, such that its Fourier transform $\widehat \varphi\in \cS(\RR)$ is the restriction of an entire function.
In particular, the argument does not apply to $\varphi \in \reallywidehat{\Ccinf(\RR)}$, since its Fourier transform vanishes on some non-empty open set.   \exend
\end{remark}

\begin{proof}
Since $\varphi \in \Ccinf(\RR)$, its Fourier transform $\widehat{\varphi}$ is the restriction of an entire function, as can be seen from standard estimates. In particular, its zero set 
\[
A\, :=\,  \{ x \in \RR \,:\, \widehat{\varphi}(x) =0 \} 
\]
is discrete without accumulation points and, therefore, locally finite and countable. The set $B:= \{y^{}_n \ts :\ts  n\in \NN \}$ is also countable, and hence so is $B-A$. Therefore, $\RR \backslash (B-A)$ is non-empty. Fix any $y \in \RR \backslash (B-A)$. We then have $\widehat{\varphi}(y^{}_n-y) \neq 0$ for all $n\in \NN$.
Since $\widehat{\varphi}$ is continuous, we find sequences $(c^{}_{n})^{}_{n\in\NN}$ and $(\delta^{}_{n})^{}_{n\in\NN}$ of positive numbers in (0,1) such that
\begin{equation}
\label{eq7}
|x-y^{}_n|\, <\, \delta^{}_n 
\, \Longrightarrow \, \left| \widehat{\varphi}(x-y) \right| \,>\, c^{}_n 
\end{equation}
for all $n\in\NN$.
Now, take any sequence $(a^{}_n)_{n\in \NN}$ of positive numbers such that
\begin{equation}
\label{eq:zstruc}
\sum_{n=1}^\infty \left(a^{}_n \cdot c^{}_n\right)^2 \,=\, \infty \, . 
\end{equation}

Next, take a decreasing sequence $(t^{}_n)^{}_{n\in \NN}$ of irrational numbers in $(0,1)$ such that $t^{}_n<\delta^{}_n$ for every $n$, such that $t^{}_1, \ldots, t^{}_n$ are linearly independent over $\QQ$ for every $n$, and such that
\begin{equation}
\label{eq4}
\sum_{n\in\NN} 2^{4n} a^{}_{n}t^{}_{n} \,< \, \infty \,.
\end{equation}

Define
\[
\sigma \,:=\,  \sum_{k=1}^\infty a^{}_k(\tau^{}_{t^{}_k}\sigma^{}_k - \sigma_{k}) \qquad \mbox{and} \qquad 
\Sigma\, :=\, \lambda \sum_{k=1}^\infty a^{}_k(\ee^{2 \pi \ii x \cdot t^{}_k}-1) \sigma^{}_k 
\]
Then $\sigma$ and $\Sigma$ are tempered measures with locally finite support by Remark~\ref{rem:stchoice}. Moreover, as tempered distributions, we have $
\widehat{\sigma}=\Sigma$. 
Thus, both $\sigma$ and $\Sigma$ are crystalline measures. Finally, recall from Proposition~\ref{cor1}  that
\begin{equation}
\label{eq5}
\supp(\sigma)\, \subset\, \QQ\cup \big(\bigcup_{n} \, (-t^{}_{n}+\QQ) \big) \qquad \mbox{and} \qquad
\supp(\Sigma)\, \subset \, \QQ \,. 
\end{equation}
Since $\sigma$ and $\Sigma$ are crystalline measures, so is $\eta$. By construction, $\eta$ is a Fourier eigenmeasure satisfying $\widehat \eta= \lambda \eta$.
Now, let $n\in \NN$ be arbitrary. 
Then, by our choice of $k^{}_n, y^{}_n$ and $t^{}_{n}$, 
we have for every $m>n$ that
\[
|\sigma^{}_{m}|\bigl([y^{}_n-1, y^{}_n+1)\bigr)\, =\, 0 \,.
\]
We thus have
\begin{align*}
\sigma\bigl(\{y^{}_{n}-t^{}_n\}\bigr)&\, =\,  \sum_{k=1}^n a^{}_k\bigl(\tau^{}_{t^{}_k}\sigma^{}_k\bigl(\{y^{}_n-t^{}_{n}\}\bigr) - \sigma^{}_{k}\bigl(\{y^{}_n-t^{}_{n}\}\bigr)\bigr) \\
&\,=\, \sum_{k=1}^n a^{}_k\bigl(\sigma^{}_k\bigl(\{y^{}_n-t^{}_n+t^{}_k\}\bigr)- \sigma^{}_{k}\bigl(\{y^{}_n-t^{}_{n}\}\bigr)\bigr) \,.
\end{align*}
Consider any $k<n$ next. Since $y^{}_n \in \QQ$, $t^{}_{n}\notin \QQ$, $t^{}_{n}-t^{}_{k} \notin \QQ$ and $\supp(\sigma^{}_k) \subset \QQ$, we have
\[
\sigma^{}_k\bigl(\{y^{}_n-t^{}_{n}+t^{}_{k}\}\bigr)=\sigma^{}_k\bigl(\{y^{}_n-t^{}_{n}\}\bigr) \,=\,0  \,.
\]
As $\sigma^{}_{n}\bigl(\{y^{}_n\}\bigr)=1$, we get
\begin{align*}
\sigma(\{y^{}_n-t^{}_{n}\})&\, =\, a^{}_n \ .
\end{align*}
Note also that Equation~\eqref{eq5} gives
\[
\Sigma\bigl(\{ y^{}_{n}-t^{}_{n} \}\bigr) \,= \, \sigma\bigl(\{ -y^{}_{n}+t^{}_n \}\bigr)\, =\, \Sigma\bigl(\{ -y^{}_{n}+t^{}_n \}\bigr) \, =\, 0  \,.
\]
For the measure $\eta$, this translates to
\[
\eta\bigl(\{y^{}_n-t^{}_{n}\}\bigr)\, =\, \sigma\bigl(\{y^{}_n-t^{}_{n}\}\bigr)\, =\,  a^{}_n  \,.
\]
Finally, since $0< t^{}_{n} < \delta^{}_{n}$, Equation~\eqref{eq7} gives 
\[
\left| \widehat{\varphi}(y^{}_n-t^{}_{n}-y) \right| \,>\, c^{}_n \,.
\]
Choose now $\psi(x):= \ee^{2 \pi \ii xy} \ts \varphi(x)$. If $\eta\ast\psi$ were Bohr almost periodic, we could use Lemma~\ref{lem-bohr} for the Fourier--Bohr coefficients of $\eta\ast\psi$ 
and obtain the estimate
\begin{align*}
 \sum_{n\in \NN}  \left|\mathbf{a}_{y^{}_n-t^{}_{n}} (\eta\ast\psi) \right|^2&\, =\, 
 \sum_{n\in \NN} 
 \bigl|\eta(\{y^{}_n-t^{}_{n}\})\bigr|^2 \cdot
 \bigl|\widehat{\psi}(y^{}_n-t^{}_{n}) \bigr|^2
 \\
 &\, =\, \sum_{n\in \NN}\bigl|\eta(\{y^{}_n-t^{}_{n}\})\bigr|^2 \cdot \bigl|\widehat{\varphi}(y^{}_n-t^{}_{n}-y)  
 \bigr|^2 \\
 &\,>\,  \sum_{n\in \NN} \left(a_n \cdot c_n\right)^2 \, =\, \infty \,.
\end{align*}
But this contradicts the finiteness statement in Lemma~\ref{lem-bohr}. Therefore, we conclude that $\eta*\psi \notin SAP(\RR)$, which means that $\eta$ is not strongly almost periodic as a distribution. Also, $\eta$~cannot be translation bounded as a distribution. Indeed, if we assume by contradiction that $\eta$ is translation bounded, Theorem~\ref{prop:M-theorem} implies that $\eta$ is a strongly almost periodic distribution, which we already have ruled out above. As $\eta$ is not translation bounded as a distribution, it fails to be translation bounded as a tempered distribution. By Lemma~\ref{lem:temptb}, this implies that $\eta$ is not strongly tempered.
\end{proof}

\section{Another exotic crystalline measure}\label{sec:NFQC2}

We have argued that a natural subclass of crystalline measures $\mu$ is such that both $\mu$ and $\widehat \mu$ are translation-bounded tempered distributions. We have also argued that this class contains all positive crystalline measures and all Fourier quasicrystals. In this section, we construct a crystalline measure in that class that fails to be a Fourier quasicrystal, using the periodic Fourier eigenmeasures from Lemma~\ref{lem:eigenmesures}. In fact, its Fourier transform is a translation-bounded measure that is even norm almost periodic. Our examples thus illustrate that the realm of crystalline measures provides unexpected phenomena beyond the class of Fourier quasicrystals.  

\smallskip

We first recall a construction by Kahane and Salem \cite{KS56}, which was also employed in this context in Example~4.7 from \cite{BS23}. It enables one to construct a sequence of finite pure-point measures on the real line, such that their total variation grows faster than exponentially and, at the same time, the norm of their Fourier transforms decays exponentially fast. 
To construct such measures, consider first for any positive numbers $a,b$ the finite pure-point measure 
\[
\omega^{}_{a,b} \, := \, \delta^{}_{0} + \delta^{}_{a}+ \delta^{}_{b} - \delta^{}_{a+b} \ .
\]
Its total variation satisfies $\| \omega^{}_{a,b} \| = 4$, and its Fourier transform satisfies $\|\widehat{\omega^{}_{a,b}}\|^{}_{\infty} \leq 2\sqrt{2}$. 

\smallskip

Next, choose sequences $(a^{}_n)^{}_{n\in \NN}$, $(b^{}_n)^{}_{n\in \NN}$ and  $(c^{}_n)^{}_{n\in \NN}$ of numbers in $(0,1)$, such that
\[
0\, <\, a^{}_1, \ldots, a^{}_n, b^{}_1,\ldots, b^{}_n, c^{}_n\, <\,  \myfrac{1}{2n+1}
\]
for every $n\in \NN$, and such that all numbers are irrational and linearly independent over $\QQ$. 
We fix $n\in\NN$ for the moment and consider the $n$-fold convolution measure
\[
{\vartheta}^{}_n \, = \, \bigast_{i=1}^n \omega^{}_{a^{}_{i},b^{}_{i}}\,.  
\]
Note that by linear independence, the finite set 
\[
F^{}_n \,:=\,\supp(\vartheta^{}_n) \, = \,  \Bigl\{ \sum_{i=1}^n \alpha^{}_{i}a^{}_i + \beta^{}_i b^{}_{i} \, : \, \alpha^{}_{i}, \beta^{}_{i} \in \{0,1\}\Bigr\} \,\subset\,  \Bigl[0,1-\myfrac{1}{2n+1}\Bigr) \ .
\]
has $2^{2n}$ elements. Moreover, there exists some $\eps^{}_n : F^{}_n \to \{ -1,1 \}$ such that 
\[
\vartheta^{}_{n} \, =\,  \sum_{x \in F^{}_n} \eps^{}_n(x) \delta^{}_x \,.
\]
It follows that $\| \vartheta^{}_{n} \| = 2^{2n}$ and, by the convolution theorem, that $\| \widehat{\vartheta^{}_{n}} \|^{}_{\infty} \leqslant 2^{\frac{3n}{2}}$.
Now,  for each $m \in \NN$, pick some $n$ such that $2^{\frac{n}{2}} > 2^m(m^2 + 1)^m$, and define the finite pure-point measure
\[ 
\varOmega^{}_{m} \, := \, \myfrac{\delta^{}_{c^{}_n}\ast \vartheta^{}_{n}}{2^m \| \widehat{\vartheta^{}_{n}} \|^{}_{\infty}} \ .
\]
This choice ensures that 
\[\| \varOmega^{}_{m} \| \, \geq \, (m^2 + 1)^m \qquad \mbox{and} \qquad  \| \widehat{\varOmega^{}_{m}} \|^{}_{\infty} \, = \, 2^{-m}\, .  \]
Moreover, we have $\supp(\varOmega^{}_m) \subset [0,1)$ and
\[
|\supp(\varOmega^{}_m)| \,=\,  |\supp(\vartheta^{}_n)| \, = \,  
4^n \,>\, 2^{\frac{n}{2}} \,>\, 2^m(m^2 + 1)^m \,.
\]
A variant of the construction in Remark~\ref{rem:choice_y} {\rm (iv)} now yields the following example of a~crystalline measure $\sigma$ that is not a Fourier quasicrystal. 

\begin{theorem}
\label{thm3.41}
Choose measures $(\varOmega^{}_m)^{}_{m\in \NN}$ as above, periodic eigenmeasures $(\mu^{}_{m})^{}_{m\in \NN}$ to eigenvalue $\lambda$ from Lemma~\ref{lem:eigenmesures}, and define
\[
\sigma\,:= \, \sum_{m=4}^{\infty} \varOmega^{}_{m} \ast \mu^{}_{m}\qquad \mbox{and} \qquad
\Sigma\,:= \,  \lambda \sum_{m=4}^{\infty} \widehat{\varOmega^{}_{m}} \cdot \mu^{}_{m} \ .
\]
Then, both $\sigma$ and $\Sigma$ are pure-point measures having locally finite support, with the following properties:
\begin{itemize}
\item[(a)] $\Sigma$ is translation-bounded as a measure, and hence a strongly tempered measure.
\item[(b)] $\sigma$ is a tempered measure but not a strongly tempered measure.
\item[(c)] As tempered distributions, we have $\widehat{\sigma}= \Sigma$.
In particular, $\sigma$ is Fourier transformable as a measure. 
\item[(d)] $\sigma$ is strongly almost periodic as a tempered distribution. In particular, it is translation bounded as a tempered distribution. 
\item[(e)] $\Sigma$ is norm almost periodic measure. In particular, $\Sigma$ is strongly almost periodic both as a tempered distribution and as a measure. 
\end{itemize}
\end{theorem}

\begin{remark}
For the measure $\sigma$ in the above theorem, set $\eta=\sigma+ \cF(\sigma)+\cF^2(\sigma)+\cF(\sigma^3)$. Then, $\eta$ is a crystalline measure that is a Fourier eigenmeasure to eigenvalue $\lambda$. It is a~tempered measure that is not strongly tempered, and also translation unbounded as a~measure. On the other hand, it is translation bounded as a tempered distribution. It thus falls into the same class as Guinand's measure with respect to its almost periodicity, but it is not a Fourier quasicrystal. \exend
\end{remark}

\begin{proof}
We will proceed similarly to the proof of Proposition~\ref{cor1}. For each $M\in \NN$, define the approximants 
\begin{equation}
\label{eq:sNSN}
\sigma^{}_{M}\,:=\, \sum_{m=1}^{M} \varOmega^{}_{m} \ast  \mu^{}_{m} \qquad \mbox{and} \qquad
\Sigma^{}_{M} \,:=\,  \lambda \sum_{m=1}^{M} \widehat{\varOmega^{}_{m}} \cdot \mu^{}_{m} \ ,
\end{equation}
where we have set $\mu^{}_m=0$ for $m\leq 3$ in order to simplify notation.
Then, both $\sigma^{}_{M}$ and $\Sigma^{}_{M}$ are translation-bounded pure-point measures having uniformly discrete support. Also, for each $R>0$ there exists some $M^{}_0$ so that for all $m>M^{}_0$ the restrictions of $\varOmega^{}_{m}\ast \mu^{}_{m}$ and $\widehat{\varOmega^{}_{m}} \cdot \mu^{}_{m}$ to $[-R,R\ts ]$ are both zero. It follows immediately that $\sigma^{}_{M}$ and $\Sigma^{}_{M}$ converge vaguely to some pure-point measures $\sigma', \Sigma'$ with locally finite support. Moreover, for each $\varphi \in \Cc(\RR)$, we have
\[
\sigma'(\varphi)\,= \, \sum_{m=1}^{\infty} \bigl(\varOmega^{}_{m} \ast \mu^{}_{m}\bigr)(\varphi) \qquad \mbox{and} \qquad
\Sigma'(\varphi)\, =\,  \lambda \sum_{m=1}^{\infty} \bigl( \widehat{\varOmega^{}_{m}} \cdot \mu^{}_{m}\bigr)(\varphi) \ ,
\]
with the sums on the right-hand side being non-zero for only finitely many terms. Thus $\sigma=\sigma'$ and $\Sigma=\Sigma'$, which proves the first claim. We now prove the stated properties of $\sigma$ and $\Sigma$.

\noindent \textbf{Ad (a):} 
Translation boundedness of $\Sigma$ follows from the estimate (analogous to the one in the proof of Proposition~\ref{prop:countreex_simple})
\[
\|\Sigma\|^{}_{(0,1)}\,\leq\, \sum_{m=1}^\infty \|\widehat{\varOmega^{}_{m}} \cdot \mu^{}_{m}  \|^{}_{(0,1)}
\, \leq \, \sum_{m=1}^\infty \|\widehat{\varOmega^{}_{m}}\|^{}_\infty \cdot \|\mu^{}_{m}  \|^{}_{(0,1)} \, \leq \, \sum_{m=1}^\infty 2^{-m} m \,=\,2 \, .
\]
The strong temperedness then directly follows, as discussed prior to Lemma~\ref{lem:tbchar}, proving~(a). And as a by-product, we obtain norm-convergence of the approximants
\begin{equation}
\label{eq2}
\lim_{M \to \infty} \| \Sigma-\Sigma^{}_{M} \|^{}_{(0,1)} \, = \, 0 \,.
\end{equation}

\noindent \textbf{Ad (b) and (c):} We show that the measure $\sigma$ is tempered. Since $\Sigma$ is a translation-bounded measure, it can be identified with a tempered distribution. Therefore, there exists a tempered distribution $T \in \cS(\RR)$ such that $\widehat{T}=\Sigma$.
Further, note that Equation~\eqref{eq2} implies $\Sigma^{}_{M} \to \Sigma$ in the topology of tempered distributions (for the corresponding estimate, see for example \cite[Lemma~4.5]{SS21}). 
Finally, the Fourier transform of the tempered measure $\sigma^{}_{M}$ is given by
\[
\reallywidehat{\sigma^{}_{M}}\, =\,  \sum_{m=1}^{M} \reallywidehat{ \varOmega^{}_{m} \ast \mu^{}_{m}} \,= \, \lambda  \sum_{m=1}^{M} \widehat{\varOmega^{}_{m}} \cdot \mu^{}_{m} \,=\,  \Sigma^{}_{M} \,.
\]
Therefore, we have for all $\varphi \in \Ccinf(\RR)$ that
\[
\sigma(\varphi)\, =\,  \lim_{M \to \infty} \sigma^{}_{M}(\varphi)\, =\,  \lim_{M \to \infty}  \Sigma^{}_{M}(\widecheck{\varphi}) \,=\,   \Sigma(\widecheck{\varphi})\,=\, T(\varphi) \,.
\]
Since $T$ is a tempered distribution, it follows that the measure $\sigma$ is tempered. Moreover, we have $\widehat{\sigma}=\widehat{T}= \Sigma$ in the tempered measure sense, which proves (c).

We now show that the measure $\sigma$ is not strongly tempered. Observe that our choice of $(a^{}_n)^{}_{n\in \NN}$, $(b^{}_n)^{}_{n\in \NN}$ and $(c^{}_n)^{}_{n\in \NN}$ implies for $m \neq n$ that 
\[
\bigl(\supp(\varOmega^{}_m)+ \QQ\bigr) \, \cap \,  \bigl(\supp(\varOmega^{}_n)+\QQ\bigr)\, =\, \varnothing \,.
\]
Thus, the pure-point measures $(\varOmega^{}_{m}\ast \mu^{}_{m})^{}_{m\in \NN}$ have pairwise disjoint supports, that is, we have for all $m \neq n$ that 
\begin{equation}
\label{eq6}
\supp( \varOmega^{}_{m}\ast \mu^{}_{m}) \, \cap \, \supp(\varOmega^{}_{n} \ast \mu^{}_{n} )\, =\,  \varnothing \,.
\end{equation}
Next, take a sequence $(y^{}_m)^{}_{m\in \NN}$ such that $0\leq  y^{}_m \leq m$ and 
$\mu^{}_m\bigl(\{y^{}_m\}\bigr) = 1$ for all $m\in \NN$. Abbreviate $S^{}_m=\supp(\varOmega_m)$ for $m\in \NN$.
By our choice of $(a^{}_n)^{}_{n\in \NN}$, $(b^{}_n)^{}_{n\in \NN}$ and $(c^{}_n)^{}_{n\in \NN}$, we have 
\[
(\varOmega^{}_m \ast \mu^{}_m)\bigl(\{x+y^{}_m\}\bigr)\,=\, \pm 1 
\]
for all $x \in S^{}_m$. Moreover, for all $n \neq m$, we have 
\[
(\varOmega^{}_n \ast \mu^{}_n)\bigl(\{x+y^{}_m\}\bigr)\,=\, 0
\] 
for all $x \in S^{}_m$ by Equation~\eqref{eq6}.
Since $S^{}_m\subset [0,1)$ for all $m$, we thus have 
\[
|\sigma|\bigl(B^{}_{m+1}(0)\bigr)
\, \geq \,
\sum_{x \in S^{}_m}\left|\mu^{}_m*\varOmega^{}_m\right|\bigl(\{x+y^{}_m\}\bigr)  \,= \, |S^{}_m| \, >\,  2^m(m+1)^2 \,. 
\]
This shows that $\mu$ is not slowly increasing, and hence not strongly tempered, which proves~(b).

\noindent \textbf{Ad (d):} The claim is immediate. Indeed, since $\widehat{\sigma}=\Sigma$ is a strongly tempered pure-point measure, we obtain that $\sigma$ is a strongly almost periodic tempered distribution from  Corollary~\ref{theo:Fav24b}.

\noindent \textbf{Ad (e):} We argue as in the proof of Proposition~\ref{prop:countreex_simple}. For each $m\in \NN$, the measure $\mu^{}_m$ is fully periodic with lattice support. Next, $\widehat{\varOmega^{}_{m}}$ is a trigonometric polynomial and hence a~Bohr almost periodic function. Therefore, $\widehat{\varOmega^{}_{m}} \cdot \mu^{}_m$ is a strongly almost periodic measure \cite[Thm.~6.1]{ARMA}. Then also the linear combination $\Sigma^{}_M$ in Equation~\eqref{eq:sNSN}
is a strongly almost periodic measure supported inside the lattice $\frac{1}{N!}\ZZ$ and hence norm almost periodic. Then, by Equation~\eqref{eq2}, the measure $\Sigma$ inherits norm almost periodicity from $(\Sigma^{}_M)^{}_{M\in \NN}$. 
In particular, $\Sigma$ is a~strongly almost periodic measure \cite[Lem.~7]{BM04}.
Furthermore, Lemma~\ref{lem:TBMimpTBTD} implies that $\Sigma$ is also a~strongly almost periodic tempered distribution.
\end{proof}

\section*{Acknowledgements}
The authors greatly acknowledge the hospitality of the Mathematisches Forschungsinstitut Oberwolfach, where the major part of this paper was written during their joint stays in March 2026. J.M. was supported by MFO as the Leibniz Fellow (Project Nr 2603q). J.M. also acknowledges support from the Deutsche Forschungsgemeinschaft (DFG, German Research Foundation), through grant MA~12254/1 (Project Nr 579756950). N.S. was supported by the NSERC with Discovery grant 2024-04853.

\end{document}